\documentclass[10pt]{amsart}
\usepackage{amsmath}
\usepackage{amsfonts}
\usepackage{amssymb}
\usepackage{graphicx}
\usepackage{amsthm,graphicx,color,yfonts}
\usepackage{pdfsync}
\usepackage{epstopdf}%
\usepackage{pifont}
\usepackage{bbm}

\usepackage[colorlinks=true]{hyperref}
\hypersetup{linkcolor=red,citecolor=blue,filecolor=dullmagenta,urlcolor=darkblue} 

\newcommand{\RR}{{{\mathbb R}}}

\newcommand{\R} {\mathbb R}
\newcommand{\cuad}{{\sqcap\kern-.68em\sqcup}}

\newcommand{\ve}{\varepsilon}

\newcommand{\be}{\begin{equation}}
\newcommand{\ee}{\end{equation}}

\newcommand{\la}{\lambda}

\definecolor{darkgreen}{rgb}{0.2,0.7,0.1}

\newcommand{\sech}{\mathop{\mbox{\normalfont sech}}\nolimits}

\newcommand{\al}{\alpha}

\def\bm{\left( \begin{array}{cc}}
\def\endm{\end{array}\right)}

\newcommand{\ba}{\begin{equation*}}
\newcommand{\ea}{\begin{equation*}}
\newcommand{\bea}{\begin{eqnarray}}
\newcommand{\eea}{\end{eqnarray}}
\newcommand{\bee}{\begin{eqnarray*}}
\newcommand{\eee}{\end{eqnarray*}}
\newcommand{\ben}{\begin{enumerate}}
\newcommand{\een}{\end{enumerate}}

\numberwithin{equation}{section}

\newtheorem{theorem}{Theorem}[section]
\newtheorem*{theorem*}{Theorem}
\newtheorem{proposition}{Proposition}[section]
\newtheorem{corollary}{Corollary}[section]
\newtheorem{lemma}{Lemma}[section]

\newtheorem{claim}{Claim}[section]

\theoremstyle{remark}
\newtheorem{remark}{Remark}[section]

\title[Kink dynamics]{Kink dynamics in the $\phi^4$ model: asymptotic stability for odd perturbations in the energy space}

\author{Micha{\l}  Kowalczyk}
\address{Departamento de Ingenier\'{\i}a Matem\'atica and Centro
de Modelamiento Matem\'atico (UMI 2807 CNRS), Universidad de Chile, Casilla
170 Correo 3, Santiago, Chile.}
\email {kowalczy@dim.uchile.cl}
\thanks{M. Kowalczyk was partially supported by Chilean research grants Fondecyt 1130126, Fondo Basal CMM-Chile and {ERC 291214 BLOWDISOL}. The author would like to thank Centre de math\'ematiques Laurent Schwartz at the  Ecole Polytechnique and the Universit\'e Cergy-Pointoise where part of this work was done.}

\author{Yvan Martel}
\address{Centre de math\'ematiques Laurent Schwartz (UMR 7640 CNRS), Ecole polytechnique, 91128 Palaiseau Cedex, France}
\email{yvan.martel@polytechnique.edu}
\thanks{Y. Martel was partially supported by ERC 291214 BLOWDISOL}

\author{Claudio Mu\~noz}
\address{CNRS and Departamento de Ingenier\'{\i}a Matem\'atica and Centro
de Modelamiento Matem\'atico (UMI 2807 CNRS), Universidad de Chile, Casilla
170 Correo 3, Santiago, Chile.}
\email{claudio.munoz@math.u-psud.fr, cmunoz@dim.uchile.cl}
\thanks{C. Mu\~noz would like to thank the Laboratoire de Math\'ematiques d'{}Orsay where part of this work was completed. His work was partly funded by Chilean research grants FONDECYT  1150202, Fondo Basal CMM-Chile, and Millennium
Nucleus Center for Analysis of PDE NC130017}


\begin{document}

\begin{abstract}
We consider a classical equation known as the $\phi^4$ model in one space dimension.
The kink, defined by $H(x)=\tanh(x/{\sqrt{2}})$, is 
an explicit stationary solution of this model. From a  result of Henry, Perez and Wreszinski \cite{MR678151}  it is known that  the kink is orbitally stable with respect to small perturbations of the initial data in the energy space.
In this paper we show asymptotic stability of the kink for odd perturbations in the energy space. The proof is based on Virial-type estimates partly inspired from previous works of Martel and Merle on asymptotic stability of solitons for the generalized Korteweg-de Vries equations (\cite{MR1753061}, \cite{zbMATH01631995}).
However, this approach has to be adapted  to additional difficulties, pointed out by Soffer and Weinstein \cite{MR1681113} in the case of general Klein-Gordon equations with potential: the interactions of the so-called internal oscillation mode with the radiation, and the different rates of decay of these two components of the solution in large time.
\end{abstract}

\maketitle

\section{Introduction}
\subsection{Main result}
In this paper we consider a classical nonlinear equation known as the $\phi^4$ model, often used in quantum field theory and other areas of physics. We refer the reader for instance to  \cite{MR2282481,MR1402248,MR2318156,MR2068924,MR1446491}.
In one space dimension, this equation writes
\begin{equation}\label{wave ac}
\partial_t^2\phi  - \partial_x^2\phi = \phi-\phi^3, \quad (t,x)\in \RR\times\RR.
\end{equation}
Recall that 
\be\label{H}
H(x)=\tanh\left(\frac{x}{\sqrt{2}}\right),
\ee
is a time-independent solution of \eqref{wave ac}, called
the {\it  kink}.  Indeed, $H$ is the unique (up to multiplication by $-1$), bounded, odd solution   of the equation
\begin{equation}
- H'' = H-H^3.
\label{hetero}
\end{equation}
Note also that \eqref{wave ac} is invariant under time and space translations and under the Lorentz transformation.
Written in terms of the pair $(\phi, \partial_t \phi)$, another important property of (\ref{wave ac}) is the fact that the \emph{energy}  
\be\label{Energy0}
E(\phi,\partial_t \phi):=\int \frac{1}{2}|\partial_t \phi|^2+\frac{1}{2}|\partial_x\phi|^2+\frac{1}{4}\left(1-|\phi|^2\right)^2,
\ee
is formally conserved along the flow. Note that the energy of the kink $(H,0)$ 
is finite and thus  $H^1\times L^2$ perturbations of the kink are referred as \emph{perturbations in the energy space}. By standard arguments, the model \eqref{wave ac} is  locally well-posed for  initial data 
$(\phi(0),\partial_t \phi(0))$ of the form $(H+\varphi_0^{in},\varphi_1^{in})$ where $(\varphi_0^{in},\varphi_1^{in}) \in H^1\times L^2$. Note also that for odd initial data, the solution of \eqref{wave ac} is odd.

As for the long time behavior of solutions of \eqref{wave ac}, we recall that Henry, Perez and Werszinski \cite{MR678151} proved orbital stability of kink with respect to small perturbations in the energy space (see Proposition \ref{henry} and its proof for the special case of odd perturbations). 
For the rest of this paper, we   work in such framework, and we   consider only odd perturbations.

Set
\be\label{eqvarphi0}
 \phi = H + \varphi_1 ,\quad   \partial_t \phi   = \varphi_2,\quad 
\varphi(t) = \left(\begin{aligned}
& \varphi_1(t) \\
& \varphi_2(t)
\end{aligned}\right).
\ee
Then, $\varphi$ satisfies
\begin{equation}\label{eqvarphi}\left\{\begin{aligned}
& \partial_t \varphi_1 = \varphi_2 \\ 
& \partial_t \varphi_2 = - \mathcal L \varphi_1 - (3 H \varphi_1^2 + \varphi_1^3), 
\end{aligned}\right.\end{equation}
where $\mathcal L$ is the linearized operator around $H$:
\begin{equation}\label{defL}
\mathcal L = - \partial_{x}^2 - 1 + 3 H^2 = - \partial_{x}^2 + 2 - 3 \sech^2\left(\frac x{\sqrt{2}}\right).
\end{equation}

Our main result is the  asymptotic stability of the kink of the $\phi^4$ model with respect to odd perturbations in the energy space.

\begin{theorem}\label{TH1}
There exists $\ve_0 >0$ such that, for any $\ve\in (0,\ve_0)$ and  for any odd $\varphi^{in}\in H^1\times L^2$ with
\[
\|\varphi^{in} \|_{H^1\times L^2}< \ve,
\]
the  global solution  $\varphi$ of (\ref{eqvarphi})  with initial data $\varphi(0)=\varphi^{in}$
satisfies
\be\label{Conclusion_0}
\lim_{t \to \pm\infty}   \|\varphi(t)\|_{H^1(I)\times L^2(I)} =0,
\ee 
for any bounded interval  $I\subset \RR$.
\end{theorem}

\medskip

To our knowledge, Theorem \ref{TH1} is the first result on the asymptotic stability of the kink for the standard one dimensional $\phi^4$ model, 
which we see as a classical question in the field. Several previous related results, outlined Section \ref{related res} below,  suggest that there have been several attempts to solve this problem by various techniques.
Moreover, a corollary of Theorem \ref{TH1} is  that no wobbling kinks (periodic in time, topologically nontrivial solitons) close to the kink exist, which partially settles  another longstanding  open question in the field (see Remark \ref{sg kink}). 
Finally, we believe that our approach, elementary and self-contained, is at the same time  general and flexible  and opens a new way to prove similar results for related models.

\medskip 

\begin{remark} We   comment here on the notion of asymptotic stability introduced in Theorem \ref{TH1}.
Observe that if a solution $\varphi$ of \eqref{eqvarphi} satisfies 
$\lim_{t \to +\infty}   \|\varphi(t)\|_{H^1 \times L^2} =0$
then by the orbital stability resut \cite{MR678151}, $ \varphi(t)\equiv 0$ for all $t\in \R$.
Thus, the notion of ``local'' asymptotic stability in the energy space as in \eqref{Conclusion_0} is in some sense optimal.
\end{remark}

The statement of  Theorem \ref{TH1} does not contain any information on the decay rate of $\|\varphi(t) \|_{H^1(I)\times L^2(I)}$ as $t\to \pm \infty$. 
To give a precise answer to this question from our proof, we need to introduce a decomposition of $\varphi(t)$ along the discrete and continuous parts of the spectrum of  $\mathcal L$ in \eqref{defL}, which respectively  correspond to internal oscillations and radiation.
The operator $\mathcal L$ is  classical and it is well-known  (see e.g. \cite{MR922041}) that  
\[
\mathrm{spec}\, \mathcal L=\Big\{0, \frac{3}{2}\Big\}\cup [2, +\infty).
\]
The discrete spectrum consists of  simple eigenvalues $\lambda_0=0$ and $\lambda_1=\frac{3}{2}$, with $L^2$ normalized eigenfunctions, respectively given by
\be\label{phi_0}
Y_0(x):=  \frac 1 2 \sech^2\Big(\frac{x}{\sqrt{2}}\Big),\quad \langle Y_0,Y_0\rangle=1,
\ee
and
\be\label{phi_1}
Y_1(x) := 2^{-3/4}3^{1/2} \tanh\left(\frac{x}{\sqrt{2}}\right) \sech\left(\frac{x}{\sqrt{2}}\right),\quad \langle Y_1,Y_1\rangle=1.
\ee
{\color{black} (Here and below $\langle F , G \rangle  := \int  FG$).}
Note that $Y_0(x)=\frac{\sqrt{2}}2 H'(x)$ is   related to the invariance of equation \eqref{wave ac} by space translation ; since we restrict ourselves to   odd perturbations of the stationary kink, this direction will not be relevant throughout this work.
In contrast, the eigenfunction $Y_1$, usually referred as the \emph{internal mode} of oscillation of the kink is not related to any invariance. It introduces serious additional difficulties and plays a key role in the analysis of the long time dynamics of $\varphi$.
We decompose $\varphi(t)$  into the form 
\be\label{DECO}
\begin{aligned}
\varphi_1(t,x) & =z_1(t)Y_1(x)+u_1(t,x),\quad\langle u_1(t),Y_1\rangle  =0,\\
\varphi_2(t,x) & = {\lambda}^{1/2}_1 z_2(t)Y_1(x)+u_2(t,x),\quad \langle u_2(t),Y_1\rangle= 0.
\end{aligned}
\ee
{\color{black}
Theorem \ref{TH1} will be the consequence of the following global estimate.
\begin{theorem}
Under the assumptions of Theorem \ref{TH1},
\be\label{Conclusion_1}
\int_{-\infty}^{+\infty} \left(|z_1(t)|^{4}+ |z_2(t)|^{4}\right) dt+ 
\int_{-\infty}^{+\infty}\int_{\RR}  \left( (\partial_x u_1)^2  + u_1^2  + u_2^2  \right)(t,x)  e^{-c_0 |x|}{dx dt}\lesssim  \|\varphi^{in} \|_{H^1\times L^2}^2 ,
\ee
for some fixed $c_0>0$. 
\end{theorem}
}

The information given by  \eqref{Conclusion_1} on the solution $\varphi(t)$ may seem rather weak compared to other existing results of asymptotic stability, but it is not clear to us whether a stronger convergence result can be obtained for general perturbations in the energy space.
Estimate \eqref{Conclusion_1}  follows from the introduction of a new  Virial functional for \eqref{eqvarphi}. This approach based on  Virial functionals is similar and inspired by works of Martel and Merle \cite{MR1753061}, \cite{zbMATH01631995} on the asymptotic behaviour of solitons in the subcritical generalized Korteweg de Vries  equations and of Merle and Raphael \cite{MR2150386} in the context of the blow-up dynamics for the mass critical nonlinear Schr\"odinger equation. This remarkable coincidence shows a deep connection between  dispersive and wave problems of seemingly different nature, and the generality of such arguments.  However, a  new, key  feature in our approach  is to  adapt   the Virial functional to take into account the internal oscillation mode $(z_1,z_2)$ associated with the direction of $Y_1$. This mode is expected to have a slower decay rate as $t\to \pm\infty$, as suggested in \eqref{Conclusion_1} and we believe that in general 
$
\int_{-\infty}^{+\infty} (|z_1(t)|^2+|z_2(t)|^2)dt = +\infty.
$

\begin{remark}\label{re:general} 
In our opinion, the case of odd perturbations contains the most difficult and at the same time the most original aspect of the problem which is the exchange of  energy between the internal oscillations and the radiation, and the discrepancy of decay rates of the different components of the perturbation.
To address this issue we have developed new tools both for the linear and the nonlinear parts of the problem. 
Asymptotic stability  of solitons for   nonlinear Schr\"odinger equations or of kinks in the relativistic Ginzburg-Landau equation in the space of odd perturbations has also been considered previously for instance in  \cite{cuccagna_3,MR2835867}. From this point of view the oddness hypothesis in Theorem \ref{TH1} is neither new nor artificial. 
This being said, we conjecture  the asymptotic stability result  to be true for general perturbations in the energy space. 
This would require   taking into account the translation invariance of the $\phi^4$ model by modulation theory, which is standard in this type of  problems (see e.g. \cite{weinstein,buslaev,cuccagna_4,MR1753061}). We expect that to treat the general case, a more refined analysis of the Virial functional introduced in Section 4 of this paper will be needed.
\end{remark}

\begin{remark}\label{sg kink}
The sine-Gordon equation 
\[
\Box  u+\sin u=0, 
\]
 shares some qualitative properties with the $\phi^4$ model, as the existence of an
  explicit kink solution
\[
S(x) {\color{black}=4\arctan (e^x)}=4\mathrm{Arg}\,(1+ie^x), \qquad x\in \R.
\]
As a classical example of integrable scalar field equation, it also 
possesses other exceptional solutions, among them a one parameter family of {\rm{odd, periodic}} solutions called \emph{wobbling kinks}, given explicitly by
(see Theorem 2.6 in \cite{CQS})
\be\label{wobbler}
W_{\alpha}(t,x)=4\mathrm{Arg}\,\big(U_\alpha(t,x)+i V_\alpha(t,x)\big),
\ee
where  
\[
\begin{aligned}
U_\alpha(t,x)&=1+\frac{1+\beta}{1-\beta} e^{\,2\beta x}-\frac{2\beta}{1-\beta}e^{\,(1+\beta)x}\cos(\alpha t),\\
V_\alpha(t,x)&=\frac{1+\beta}{1-\beta} e^{\,x}+e^{\,(1+2\beta)x}-\frac{2\beta}{1-\beta}e^{\,\beta x}\cos(\alpha t),
\end{aligned}
\]
for $\alpha\in (0,1)$, $\beta=\sqrt{1-\alpha^2}$. Taking $t=\frac{\pi}{2\alpha}$ it is not hard to see that
\[
\Big\|(S,0)- \Big(W_\alpha(\frac{\pi}{2\alpha}),\partial_t W_\alpha(\frac{\pi}{2\alpha}) \Big)\Big\|_{H^1\times L^2}=\mathcal O(\beta).
\]
Since, at the same time  $W_\alpha$ is periodic in time, we see that \emph{the sine-Gordon kink is not asymptotically stable in the energy space} in the sense of Theorem \ref{TH1} (take $0<\beta\ll 1$). This is a remarkable difference between the two models at the nonlinear level.

We note further that   
\[
S(x)=2\pi +  \mathcal O(e^{\,-x}), \qquad W_{\alpha}(\frac{\pi}{2\alpha},x)=2\pi + \mathcal O(e^{\,-x}), \qquad x\to\infty,
\]
with similar formulas when $x\to -\infty$. From these facts and the explicit formula for $\partial_t W_{\alpha}(\frac{\pi}{2\alpha},x)$ it is not hard to see that
$S$ and $W_\alpha(\frac{\pi}{2\alpha})$ are also close in Sobolev norms with some exponential weight. It follows that even in a stronger topology asymptotic stability  does not hold for the sine-Gordon equation, in contrast to  
results proven for example in \cite{PW} for the generalized Korteweg-de Vries equation, or in \cite{cuc_cub_nls} for the cubic one dimensional nonlinear Schr\"odinger equation. 

The problem of constructing or proving non-existence of  wobbling kinks for the $\phi^4$ model attracted some attention in the past; for an early discussion, we refer the reader to Segur's work \cite{MR708660}. While there was formal and numerical evidence against the existence of  wobbling kinks \cite{kruskal_segur}, our result provides a rigorous proof of non-existence (at least in a neighborhood of the kink) that, to our knowledge, had been missing. 

\end{remark}

\subsection{Discussion of  related results}\label{related res}

As we have seen,  the question of stability of the kink, as a solution of \eqref{wave ac}, with respect to small and odd perturbations  reduces to the stability of the zero solution of the nonlinear Klein-Gordon (NLKG) equation \eqref{eqvarphi}. Similarly, the question of \emph{asymptotic stability} for $H$, in a suitable topological space, in principle reduces to the  problem of  ``scattering''  of small  solutions for \eqref{eqvarphi}. In particular  \eqref{eqvarphi} presents several well-known difficulties: it is a \emph{variable-coefficients} NLKG equation with both nontrivial quadratic and cubic nonlinearities, in one space dimension.

\medskip

The description of the long time behaviour for small solutions to NLKG equations has attracted the interest of many researchers during the last thirty years. Klainerman \cite{K1,K2} showed global existence of small solutions in $\R^{1+3}$ via the vector field method, assuming that the nonlinearity is quadratic. Similarly, Shatah \cite{Shatah1} considered the NLKG equation with quadratic nonlinear terms in $d\geq 3$ space  dimensions. By using Poincar\'e normal norms suitably adapted to the infinite dimensional Hamiltonian system, he showed global existence for small, sufficiently regular initial data in  Sobolev spaces. In both of these approaches the main point is to deal with the quadratic nonlinearity, which causes problems even in dimension $3$ and higher due to slow rate of decay for linear Klein-Gordon waves. The situation is known to be even more delicate in low dimensions $1$ and $2$.

\medskip

In one dimension, a fundamental work due to Delort \cite{Delort} (see also \cite{MR2056833} for the two dimensional setting) shows global existence of small solutions  not only for semilinear, but also for quasilinear Klein-Gordon equations. His result was subsequently improved and simplified by Lindblad and Soffer \cite{LS1,LS2} in the semilinear case,  assuming constant coefficients with a possibly variable coefficient in front of the cubic nonlinearity \cite{LS3}.   A similar conclusion, based on a delicate analysis using Fourier decomposition methods was obtained by Sterbenz \cite{Ste}. The $\phi^4$ model serves as one of the motivations of  \cite{LS3} and \cite{Ste} since the understanding of the asymptotic stability for $H$ is deeply related to the study of NLKG equations with variable coefficients and quadratic and cubic nonlinearities. However, 
the $\phi^4$ model do not fit the assumptions made in \cite{LS3} and \cite{Ste}. Finally, we mention the works by Hayashi and Naumkin \cite{HN,HN1} on  the modified scattering procedure for cubic and quadratic constant coefficients NLKG equations in one dimension.

\medskip

{\color{black}In addition to the aforementioned difficulties of NLKG equations with  quadratic and cubic nonlinearities, what makes problem \eqref{eqvarphi} challenging is the existence of an  internal mode of oscillation}. 
The fundamental work of Soffer and Weinstein \cite{MR1681113} seems to be the first in which the mechanism of the exchange of energy between the internal oscillations and radiation was fully explained for a class of nonlinear Klein-Gordon equations with potential (see also \cite{MR1664792} by the same authors). Although the models they considered do not include the $\phi^4$ model, they speculated (see page 19 in \cite{MR1681113} that  the phenomena of slow radiation for the $\phi^4$ model is due to a similar mechanism. Since the two problems are related we will  briefly discuss their approach. They study   the question of asymptotic stability of the vacuum state  $u=0$ for the following Klein-Gordon equation in $\R^3\times \R$:
\begin{equation}\label{sw 1}
\partial_t^2 u=(\Delta-V(x)-m^2) u+\lambda u^3, \quad\lambda\in \R,~\lambda\neq 0.
\end{equation}
Under some natural hypothesis  on the decay of the potential $V$ and assuming that:
\begin{itemize}
\item[(i)]
the operator $L_V=-\Delta + V+m^2$ has  continuous spectrum $\sigma_{\rm cont}=[m, \infty)$,  a single, positive discrete eigenvalue $\Omega^2<m^2$, and the bottom of the continuous spectrum is not  a resonance;

\smallskip

\item[(ii)] the Fermi Golden Rule holds (see Remark \ref{FGrule} below);
\end{itemize}  
they show that $u=0$ is asymptotically stable. Moreover, Soffer and Weinstein proved that the internal oscillation mode decays as $\mathcal O(|t|^{-1/4})$, while the radiation decays as $\mathcal O_{L^8}(|t|^{-3/4})$. From their result we see the anomalously slow time decay rate of the  solution is additionally complicated by the existence of different time decay rates of each component of the solution. This discordance seems to hold in general, and it is also a characteristic of our problem, as expressed for instance in (\ref{Conclusion_1}). 

\medskip

The idea of the proof in \cite{MR1681113} is first to project (\ref{sw 1}) onto the discrete and the continuous parts of the spectrum of $L_V$, with the corresponding components $\eta$ and $z$ satisfying, respectively, a nonlinear dispersive equation and a Hamiltonian system. In the second step the equation for $\eta$ is solved for a given $z$  and the result then substituted in the ODE for $z$. The last step is the identification of the equation of the amplitude $A(t)=|z(t)|$ using the Poincar\'e normal forms. The second step relies on dispersive theory and this is  where the assumption (i) is used; the importance of working on $\R^3$ is evident at this point since dispersive estimates improve with the dimension of space so that the estimates for the nonlinear terms can be closed.  

\medskip

As for problems more closely related to our result we   mention the work of  Kopylova and Komech \cite{MR2835867} (see also  \cite{MR2770013}) where the issue of asymptotic stability of the kink in the following relativistic Ginzburg-Landau equation is addressed
\begin{equation}
\label{kk 1}
\partial_t^2 u=\partial_x^2 u +F(u), \quad \mathrm{in}\ \R\times \R.
\end{equation}
The form of the  nonlinearity $F=-W'$, where $W$ is a smooth  double well potential, guarantees existence of a kink $U$. More specifically it is assumed that, with some $m>0$ and $a>0$:
\begin{equation}
\label{kk 2}
W(u)=\frac{m^2}{2} (u\mp a)^2+\mathcal O(|u\mp a|^{14}), \qquad \mathrm{as}\ u\to \pm a.
\end{equation}
This   assumption, which excludes the $\phi^4$ model, is essential to the method of \cite{MR2835867} which is based on Poincar\'e normal forms and dispersive estimates and is inspired by \cite{buslaev} (see also \cite{bus_per1, bus_per2}). Under further  hypothesis  of the same type as (i) and (ii) above, 
\cite{MR2835867} shows asymptotic stability of the kink $U$ with respect to odd perturbations. The authors obtains explicit rates of decay: $\mathcal O(|t|^{-1/2})$ for the internal oscillations, and $\mathcal O_{E_{-\sigma}}(|t|^{-1})$, where $E_{-\sigma}=(1+|x|)^\sigma H^1\times (1+|x|)^\sigma L^2_{\sigma}$, $\sigma>5/2$, for the radiation.   We note that in these works, because of the slow decay of the solutions of the free equation in  one dimension, the perturbation of the quadratic
potential has to be taken  sufficiently flat near the limit points $\pm a$ in order to  close the nonlinear  dispersive estimates. 

\medskip

It is also important to mention that Cuccagna \cite{MR2373326} proved   asymptotic stability of  planar wave fronts in the $\phi^4$ model in $\R^3$
(study of the one dimensional kink subject to three dimensional perturbations). 
The method used in this paper combines dispersive estimates by  Weder  \cite{MR1729096,MR1736195} (see also \cite{MR2096737}), together with Klainerman vectors fields and normal forms.
The fact that the space dimension is three with better decay estimates for free solutions is essential in order to close the nonlinear estimates.

\bigskip

\section{Outline of the proof}\label{sec:2}

\medskip

The method of the present  paper, based on the use of a Virial functional, is inspired by  the one introduced for the generalized KdV equation in \cite{MR1753061,zbMATH01631995}. This approach is both self-contained and elementary. Below we present the key ideas. 

\medskip

{1. {\it Spectral decomposition}.} It is essential to decompose the solution $\varphi$ of (\ref{eqvarphi}) to separate  the mode of internal oscillations (associated with the eigenfunction $Y_1$) from the radiation (associated with the continuous part of the spectrum).
Indeed, these components have  specific asymptotic behavior as $t\to +\infty$. With  the notation
{ 
\[
\langle F , G \rangle   := \int  FG,
\]
}
we define
\begin{equation}\label{defz}
z_1(t) := \langle \varphi_1(t) , Y_1 \rangle,\quad
z_2(t) := \frac 1{\mu} \langle \varphi_2(t) , Y_1 \rangle,\quad \mu := \sqrt{\frac 32},
\end{equation}
\begin{equation}\label{defu}
u_1(t) := \varphi_1(t)- z_1(t) Y_1,\quad
u_2(t) := \varphi_2(t)- \mu z_2(t) Y_1,
\end{equation}
so that
\be\label{OrthoU}
\forall t\in \R, ~ \langle u_1(t),Y_1(t)\rangle = \langle u_2(t),Y_1(t)\rangle=0.
\ee
Denote
\begin{equation}\label{defzu}
z(t) := \left(\begin{aligned}
& z_1(t) \\
& z_2(t)
\end{aligned}\right), \quad 
u(t):= \left(\begin{aligned}
& u_1(t) \\
& u_2(t)
\end{aligned}\right).
\end{equation}
Set
\begin{equation}\label{defab}
|z|^2(t) := z_1^2(t) + z_2^2(t), \quad \alpha(t): = z_1^2(t) - z_2^2(t),\quad \beta(t) := 2 z_1(t) z_2(t),
\end{equation}
For the higher order nonlinear terms we will use the notation (we omit the dependence on time)
\[
\mathcal O_3 := \mathcal O(|z|^3,|z| \|u\|,\|u\|^2),
\]
where $ \mathcal O(\cdot)$ refers to a function which is bounded by a linear combination of its arguments. In this section, we present the argument formally, that is why we do not specify which norm of $u$ is used. In  Section~5 we will give full  details on the control of the error terms.

\medskip

Since $\mathcal L Y_1 = \mu^2 Y_1$, we obtain
\begin{equation}\left\{\begin{aligned}
& \dot z_1 = \mu z_2 \\
& \dot z_2 = - \mu z_1 -\frac 3{\mu}  \langle HY_1^2,Y_1\rangle\,  z_1^2+F_z, \quad F_z= \mathcal O_3.
\end{aligned}\right.\end{equation}
In particular,
\begin{equation}\label{13deux}
\frac d{dt} (|z|^2)   = \mathcal O_3.
\end{equation}
and
\begin{equation}\label{eqofab}\left\{\begin{aligned}
& \dot \alpha = 2 \mu \beta + F_\alpha, \quad F_\alpha = \mathcal O_3\\
& \dot \beta  = - 2 \mu \alpha+ F_\beta, \quad F_\beta=   \mathcal O_3;
\end{aligned}\right.
\end{equation}
Moreover, thanks to \eqref{defu} and \eqref{eqvarphi}, one checks that
\begin{equation}\label{eqofu}\left\{\begin{aligned}
& \dot u_1 = u_2 \\
& \dot u_2 =   - \mathcal L u_1 - 2 z_1^2 f + F_u,\quad F_u= \mathcal O_3;
\end{aligned}\right.\end{equation}
where $f$ is an  odd Schwartz function  given by the expression 
\begin{equation}\label{deff}
\begin{aligned}
f  :=\frac 32 \left( HY_1^2 - \langle HY_1^2  , Y_1\rangle Y_1 \right),
	 \quad &  \hbox{so that $\langle f , Y_1\rangle=0,$} \\
& \hbox{and $\forall x\in \RR,$
$|f(x)|+|f'(x)|\lesssim e^{-\frac {|x|}{\sqrt{2}}}$}.
\end{aligned}
\end{equation} There is a simple way to replace the term $z_1^2 f$  by a term involving only $\alpha$ without changing the structure of the problem. To this end we introduce a change of unknown
\begin{equation}\label{defv}
\begin{aligned}
v_1 (t,x)& := u_1(t,x) +   |z|^2(t) q(x), \quad \hbox{where} \quad \mathcal L q (x)=  f(x);\\
v_2(t,x) &: = u_2(t,x),
\end{aligned}
\end{equation}
and
\[
v(t)  := \left(\begin{aligned}
& v_1(t) \\
& v_2(t)
\end{aligned}\right).
\]
Note that the existence of a unique odd solution $q\in H^1(\R)$ of $\mathcal L q =  f$ follows from standard ODE arguments.  Moreover, $q$ satisfies
\begin{equation}\label{decayq}
\forall x\in \RR,\quad
|q(x)|+|q'(x)|\lesssim e^{-\frac{|x|}{\sqrt{2}}}.
\end{equation}
 Then,
\begin{equation}\label{eqofv}\left\{\begin{aligned}
  \dot v_1 &= v_2 + F_1, \quad F_1=\mathcal O_3, \\
  \dot v_2 &=
        - \mathcal L v_1 -   \alpha f  +F_2,  \quad F_2=\mathcal O_3   .
\end{aligned}\right.\end{equation}
Note that 
\[
0=\langle f,Y_1\rangle = \langle \mathcal L q,Y_1\rangle = \langle q, \mathcal L Y_1\rangle = \mu^2 \langle   q,Y_1\rangle
\]
and thus $\langle v_1, Y_1\rangle =\langle v_2, Y_1\rangle=0$.

\medskip

2. {\it{Orbital stability}}. 
Using  the stability result of Henry, Perez and Werszinski \cite{MR678151} (see Proposition \ref{henry} for a short proof of this result for odd perturbations), we infer that if $\varphi^{in}$ is small enough then $\varphi(t)$ is uniformly small in $H^1\times L^2$ and,   in particular, it is  global in time  in the energy space.
It follows that $(u_1,u_2)$, $z$ and thus $(v_1,v_2)$  and $\alpha, \beta$ are also  small, uniformly in time. 

\medskip

At this point, the system in $(v_1,v_2,\alpha,\beta)$ can be studied by Virial argument. We present the formal argument, discarding for a moment higher order terms.

\medskip

3. {\it{Virial type arguments}}. 
The ultimate goal of Virial arguments is to prove the following estimate, whose proof   requires the use of several functionals
\begin{equation}\label{13un}
\int_{-\infty}^\infty \left(|z(t)|^4+\|v(t)\|_{H^1_{\omega}\times L^2_{\omega}}^2\right) dt <\varepsilon^2.
\end{equation}

First, let
\begin{equation}
\label{def I}
\mathcal I  := \int \psi (\partial_x v_1) v_2 + \frac 12 \int \psi' v_1 v_2,
\end{equation}
\begin{equation}
\label{def J}
\mathcal J := \alpha \int v_2 g - 2 \mu \beta \int v_1 g,
\end{equation}
where $\psi$ and $g$  are functions to be chosen ($\psi$ is bounded, increasing; $\psi'$ and $g$ are Schwartz functions). Using the equations for $(v_1,v_2)$ and $\alpha$, $\beta$, we find
$$
-\frac d{dt} (\mathcal I + \mathcal J) = \mathcal B (v_1) + \alpha \langle  v_1, \tilde h\rangle + \alpha^2 \langle f,g\rangle
+\varepsilon\,  \mathcal O\left(|z|^4,\|v\|_{H^1_{\omega}\times L^2_{\omega}}^2\right).
$$
where $\mathcal B$ is a quadratic form and $\tilde h$ is a given Schwartz function.
Here, $H^1_{\omega}\times L^2_{\omega}$  means the $H^1\times L^2$ norm of $v$ with a suitable exponential weight, see \eqref{nloc} for a precise definition.

Next, one proves, assuming the orthogonality $\langle v_1, Y_1\rangle =0$ and the oddness of  $v_1$, that for a suitable choice of function $g$, the following coercivity property holds
$$
\mathcal B (v_1) + \alpha \langle  v_1, \tilde h\rangle + \alpha^2 \langle f,g\rangle \gtrsim \|v_1\|^2_{H^1_{\omega}} +\alpha^2.
$$
This is the key estimate of this paper. Note that the choice of function $g$ is related to the Fermi-Golden rule, see Section 3.2 for details.
We thus obtain
\begin{equation}\label{12un}
-\frac d{dt} (\mathcal I + \mathcal J) 
\gtrsim \alpha^2 + \|v_1\|^2_{H^1_{\omega}} 
+\varepsilon \,\mathcal O\left(|z|^4,\|v\|_{H^1_{\omega}\times L^2_{\omega}}^2\right).
\end{equation}

\medskip

Second, let $\gamma = \alpha \beta$. Then, 
\begin{equation}\label{12deux}
\dot \gamma = 2 \mu (\beta^2 - \alpha^2) +   \varepsilon \, \mathcal O\left(|z|^4,\|v\|_{H^1_{\omega}\times L^2_{\omega}}^2\right).
\end{equation}

Finally, by direct computations, it can be checked that
\begin{equation}\label{12trois}
\frac d{dt} \int \sech\left(\frac x{2\sqrt{2}}\right)v_1 v_2
\gtrsim \|v_2\|_{L^2_{\omega}}^2 + O \left(|z|^4,\|v\|_{H^1_{\omega}\times L^2_{\omega}}^2\right).
\end{equation}

Since
$
|z|^4 = \alpha^2+\beta^2,
$
we check that for $\varepsilon>0$ small enough, integrating in time a suitable linear combination of \eqref{12un}, \eqref{12deux} and \eqref{12trois} gives \eqref{13un}.

\medskip

4. {\it {Convergence to the zero state for a weighted norm}}. From \eqref{13un}, one deduces
\begin{equation}\label{13trois}
\|v_1(t)\|_{H^1_{\omega}} +\|v_2(t)\|_{L^2_{\omega}} + |z(t)| \to 0 \quad \hbox{as $t\to \infty$},
\end{equation}
which implies \eqref{Conclusion_0}.

Indeed, first, from \eqref{13un} it follows that there exists a sequence $t_n\to +\infty$ such that 
\[ \lim_{n\to +\infty}\|v(t_n)\|_{H^1_{\omega}\times L^2_{\omega}} + |z(t_n)|=0.
\]

For $z(t)$, from \eqref{13deux}, 
\[
\left| \frac d{dt} |z|^4\right|
\lesssim |z|^3 \left(|z|^2 + \|v\|_{H^1_{\omega}\times L^2_{\omega}}^2\right).
\]
Integrating on  $[t,t_n]$, passing to the limit $n\to +\infty$ and using \eqref{13un}, 
we deduce that  $\lim_{t\to +\infty} |z(t)|=0$.

For $v(t)$, we consider an   energy type quantity (at the linear level)
\[
\mathcal H (t) = \int \left( |\partial_x v_1|^2 + 2 |v_1|^2 + |v_2|^2 \right)
\sech\left(\frac x{2 \sqrt{2}}\right),
\]
and we check that
\[
|\dot{\mathcal H}(t)|\lesssim \left( |z|^4 + \|v\|^2_{H^1_{\omega}\times L^2_{\omega}} \right)
\]
As before, integrating on  $[t,t_n]$, and using \eqref{13un}, we deduce that $\lim_{t\to +\infty} \mathcal H(t)=0$, which proves \eqref{13trois}.

\medskip

From the above sketch it is evident that our approach is  different from the ones briefly described in Section \ref{related res} in that it does neither  rely on dispersive estimates nor on normal forms. In particular we do not use the non resonance condition\footnote{Since we we work in the odd energy space  the bottom of the continuous spectrum is  non resonant. However,  this fact is never used explicitly or implicitly.}. 
In this sense our method allows for minimal hypothesis, is robust and possibly applicable to other type of problems. In particular our method should be applicable to  generic, analytic, nonlinear perturbation of the sine-Gordon equation of the form $\sin u+h(u)$, $h(u)=\sum_{k=2}^\infty a_k u^{2k+1}$, $h(u-\pi)=-h(-u)$. Recall the related fact that breathers in the sine Gordon equation  do not persist under generic analytic perturbations (see \cite{denzler} and the references therein).

\section{Preliminaries}\label{PRELIM}

\subsection{Orbital stability of the kink}\label{orb stab}
We recall briefly
the proof of the stability result of Henry, Perez and Werszinski \cite{MR678151} restricted to odd perturbations. First, we have the  following energy conservation  for $\varphi$.
For all $t$ so that $\varphi(t)$ exists in the energy space and $\varphi(0)=\varphi^{in}$, it holds
\begin{equation}\label{envarphi}
\mathcal E(\varphi(t)): =   \int \varphi_2^2(t) + \langle \mathcal L \varphi_1(t), \varphi_1(t)\rangle
+2 \int H \varphi_1^3(t) + \frac 12 \int \varphi_1^4(t) =\mathcal E(\varphi^{in}).
\end{equation}
This result is a simple consequence of the energy conservation law \eqref{Energy0}.

\begin{proposition}[\cite{MR678151}]\label{henry}
There exist $C>0$ and   $\ve_0>0$  such that, for all $\ve \in (0,\ve_0)$ and for any odd $\varphi^{in}\in H^1\times L^2$, if
$
\|\varphi^{in}\|_{{H^1}\times{L^2}}< \ve,
$
the solution $\varphi$ of (\ref{eqvarphi}) with initial data $\varphi(0)=\varphi^{in}$ is global in the energy space and  satisfies  
\begin{equation}
\label{asympt stab 1}
\forall t\in \RR,\quad
\|\varphi(t)\|_{H^1\times L^2}< C \|\varphi^{in}\|_{H^1\times L^2}.
\end{equation}
\end{proposition}
\begin{proof}[Proof of Proposition \ref{henry}]
We begin with the following simple result.
\begin{claim}\label{onL} 
If $\varphi_1\in H^1(\R)$ satisfies $\langle \varphi_1,Y_0\rangle =0$, then
\[
\langle \mathcal L \varphi_1,\varphi_1 \rangle \geq \frac 37 \|\varphi_1\|_{H^1}^2.
\]
\end{claim}
\begin{proof}[Proof of Claim \ref{onL}]
By the spectral properties of $\mathcal L$ and the spectral theorem, we have immediately
\[
\langle \mathcal L \varphi_1,\varphi_1 \rangle \geq \frac 32 \|\varphi_1\|_{L^2}^2.
\]
Since ${\sech}^2\big(\frac{x}{\sqrt{2}}\big)\leq 1$,
\begin{align*}
\langle \mathcal L \varphi_1,\varphi_1 \rangle& = 
\int (\partial_x \varphi_1)^2 + 2 \int \varphi_1^2 - 3 \int {\rm sech}^2\left(\frac{x}{\sqrt 2}\right) \varphi_1^2\\
& \geq  \int (\partial_x \varphi_1)^2 +  \frac 57  \int \varphi_1^2 - \frac{12}{7} \int {\rm sech}^2\left(\frac{x}{\sqrt 2}\right) \varphi_1^2\\
& \geq \frac 37 \int (\partial_x \varphi_1)^2 -\frac 37  \int \varphi_1^2
+\frac 47 \langle \mathcal L \varphi_1,\varphi_1\rangle  \geq \frac 37 \|\varphi_1\|_{H^1}^2.
\end{align*}
This ends the proof of the Claim. 
\end{proof}
Going back to the proof of \eqref{asympt stab 1}, on the one hand,
\[
\mathcal E(\varphi^{in}) \leq   \|\varphi_2^{in}\|_{L^2}^ 2+2\|\varphi_1^{in}\|_{H^1}^2  + O(\|\varphi_1^{in}\|_{H^1}^3),
\]
and on the other hand,
\[
\mathcal E(\varphi(t)) \geq \frac 3{7}\left(\|\varphi_2(t)\|_{L^2}^ 2+\|\varphi_1(t)\|_{H^1}^2\right) - O(\|\varphi_1(t)\|_{H^1}^3).
\]
Combining these estimates with the energy conservation \eqref{envarphi}, for $\ve>0$ small enough, we get the result.
\end{proof}

\subsection{ODE arguments and the Fermi-Golden rule}
This section concerns the resolution of the equation $(-\mathcal L +4\mu^2) G  =   F$, where $\mu^2 =\frac32$.
This will be crucial in choosing the function $g$ in the definition of $\mathcal J$ in (\ref{def J}).
\begin{lemma}\label{le:ODE}
{\rm (i)} Let $F\in L^1(\RR)\cap C^1(\RR)$ be a real-valued function. The function $G\in L^\infty(\RR)\cap C^2(\RR)$ defined by
\begin{equation}\label{var par 7}
G(x)=\frac 1{12} {\rm Im} \left( k(x) \int_{-\infty}^x \bar k F + \bar k(x) \int_x^{+\infty} kF\right), \qquad 
\end{equation}
where
\begin{equation}\label{var par 5}
k(x)=e^{i2x} \left(1+\frac 12 \sech^2\left(\frac{x}{\sqrt{2}}\right)+i\sqrt{2}\tanh\left(\frac{x}{\sqrt{2}}\right)\right),
\end{equation}
satisfies
\begin{equation}
\label{var par 4}
(-\mathcal L +4\mu^2) G  =   F.
\end{equation}
{\rm (ii)} Assume in addition that $F\in \mathcal S(\RR)$. Then,
\begin{equation}\label{var par 3}
G\in \mathcal S(\RR)\quad \iff \quad \langle k ,  F\rangle=0.
\end{equation}
\end{lemma}
\begin{remark}
Since $k(-x)=\bar k(x)$, if $F$ is odd then $G$ defined by \eqref{var par 7} is also odd and the orthogonality condition in \eqref{var par 3} reduces to
\begin{equation}\label{var par 2}
\langle{\rm Im}( k ) ,F \rangle=0.
\end{equation}
\end{remark}
\begin{remark}\label{FGrule} Note that  $(-\mathcal L +4\mu^2)k = (-\mathcal L + 6) k =0$ (see the proof of Lemma \ref{le:ODE}).
Since $(-\mathcal L +\frac 32)Y_1=0$, we have
$\langle k, Y_1\rangle  = 0$. Thus, for $f$ defined in \eqref{deff},
\[
\langle k, f\rangle = \frac 32 \langle k , H Y_1^2\rangle=\frac 32 i \langle {\rm Im}(k) , H Y_1^2\rangle.
\]
The fact that 
\begin{equation}\label{fermi1}
\langle {\rm Im}(k) , H Y_1^2\rangle \neq 0,
\end{equation}
can be easily checked by numerical integration (we find  $\langle {\rm Im}(k) , H Y_1^2\rangle 
 \approx -0.222$). This is a kind of non resonance  condition and it  was identified by Sigal \cite{MR1218303} in the context of nonlinear wave and  Schr\"odingier equations and by Soffer and Weinstein \cite{MR1681113} in the context of a nonlinear Klein-Gordon equation in three dimensions. This condition is  a nonlinear version of the  \emph{Fermi Golden Rule} (from now on we will refer to it as such) whose origin is in  quantum mechanics \cite{simon_1}, \cite[p. 51]{reed-simon_4}. As it was pointed out in \cite{MR1681113}, this nonzero condition guarantees that the internal oscillations are coupled to  radiation and, as a consequence of this fact, the energy of the system eventually radiates away from the kink neighborhood, making it asymptotically  stable. This behavior is in deep contrast with the (integrable) sine-Gordon equation, for which the Fermi Golden Rule is definitely not satisfied (simply because $0$ is the only discrete eigenvalue for the associated linear operator), and even worse, explicit periodic solutions (wobblers, see \eqref{wobbler}) do persist.

\medskip

In the present paper, the Fermi Golden Rule \eqref{fermi1} is the key fact  which forbids  the existence of a solution $g$ of $(-\mathcal L + 4\mu^2) g = f$ {in the energy space}. As we will see, in our particular setting, and for purely algebraical reasons, we will use a modified version of (\ref{fermi1}), which reads 
\begin{equation}
\label{fermi 2}
\left\langle {\rm Im}(k) , \left(H Y_1^2-Y_1\langle H Y_1^2, Y_1\rangle\right)\sech^2 \Big(\frac{x}{8\sqrt{2}}\Big)\right\rangle\approx -0.218. 
\end{equation}

This condition is better adapted to the use of weighted Sobolev spaces which we work with, see Section \ref{coer D} and \eqref{def a}.   
\end{remark}

\begin{proof}[Proof of Lemma \ref{le:ODE}] (i) The explicit expression of $k$ in \eqref{var par 5} was found by Segur \cite{MR708660}. We easily check by direct computation that
\[
(-\mathcal L +4\mu^2)k = (-\mathcal L + 6) k =0.
\]
Indeed, let $m(x):=1+\frac 12 \sech^2\left(\frac{x}{\sqrt{2}}\right)+i\sqrt{2}\tanh\left(\frac{x}{\sqrt{2}}\right)$.
Then,
\begin{align*}
m'(x)&=-\frac{\sqrt 2}2\tanh\left(\frac{x}{\sqrt{2}}\right)\sech^2\left(\frac{x}{\sqrt{2}}\right)+ i \sech^2\left(\frac{x}{\sqrt{2}}\right),\\
m''(x)&=-\frac 32\sech^4\left(\frac{x}{\sqrt{2}}\right)+\sech^2\left(\frac{x}{\sqrt{2}}\right)-i\sqrt 2 \tanh\left(\frac{x}{\sqrt{2}}\right)\sech^2\left(\frac{x}{\sqrt{2}}\right),
\end{align*}
and thus
\begin{align*}
(-\mathcal L + 6) k & = [(\partial_x^2 + 4)e^{i2x}] m + 4i e^{i2x} m' + e^{i2x} m''+3\sech^2\left(\frac{x}{\sqrt{2}}\right)e^{i2x} m\\
& = e^{i2x} \left( 4im'+m''+3\sech^2\left(\frac{x}{\sqrt{2}}\right)m\right)=0.
\end{align*}
Since $k(x)\sim (1+i\sqrt{2}) e^{i2x}$ as $x\sim +\infty$, the functions $k$ and $\bar k$ form a set of independent solutions of 
$(-\mathcal L + 6) G=0$ (in fact, up to  multiplicative constants $k$ and $\bar k$ are the so-called Jost functions).
Therefore, for $F\in L^1(\RR)\cap C^1(\RR)$, the following real-valued function $G$ 
\[
G(x)=-{\rm Re}\left[ \frac 1{W(k,\bar k)} \left( k(x) \int_{-\infty}^x \bar k F + \bar k(x) \int_x^{+\infty} kF\right)\right],
\]
where $W(k,\bar k)=k\bar k'-k' \bar k$ is the Wronskian of $k$ and $\bar k$, solves $(-\mathcal L + 6) G =F$.
Since $W(k,\bar k) = -12 i,$ 
we obtain \eqref{var par 7}.

\medskip

\noindent
(ii) 
By the definition of $G$ in \eqref{var par 7}, we have the following asymptotics at $\pm \infty$:
\[
G(x) \sim \frac 1{12} {\rm Im}\left( k(x) \langle \bar k, F \rangle \right) \hbox{ as $x\to +\infty$},
\quad
G(x) \sim \frac 1{12} {\rm Im}\left( \bar k(x) \langle  k, F \rangle \right) \hbox{ as $x\to -\infty$}.
\]
Thus, $\lim_{x\to \pm \infty} G=0$ if and only if $\langle \bar k, F \rangle=0$. Moreover, if  $\langle \bar k, F \rangle=0$  and $F\in \mathcal S(\RR)$, it follows directly from \eqref{var par 7} that $G \in \mathcal S(\RR)$.
\end{proof}

\bigskip

\section{Virial type arguments}\label{VIRIAL}

\medskip

Recall the set of coupled equations (\ref{eqofab}) and (\ref{eqofv}), which can be recast as a mixed system in $(v_1,v_2,\alpha,\beta)$:
\begin{equation}
\label{syd}\left\{
\begin{aligned}
 \dot v_1 & =v_2 + F_1,\\
\dot v_2 & = -\mathcal L v_1 -  \alpha f + F_2, \\
 \dot \alpha &= 2 \mu \beta + F_\alpha, \\
 \dot \beta   & = - 2 \mu \alpha+ F_\beta.
\end{aligned}\right.
\end{equation}
where $\mathcal L = - \partial_{x}^2 + 2 - 3(1-H^2)$, $\mu= \sqrt{3/2}$, and $f$ was introduced in \eqref{deff}. For the moment, we will not make explicit the nonlinear error terms $F_1,$ $F_2$, $F_\alpha$ and $F_\beta$, which will be considered in detail in  Section \ref{ERROR}.

\subsection{Virial type identities}
For a smooth and bounded  function $\psi$ to be chosen later, let
\begin{equation}\label{defI}
\mathcal I  := \int \psi (\partial_x v_1) v_2 + \frac 12 \int \psi' v_1 v_2 = \int \left(\psi \partial_x v_1 + \frac 12 \psi' v_1\right) v_2.
\end{equation}
First, using \eqref{syd} and integrating by parts,
\begin{equation}\label{Ione}
\begin{aligned}
\frac{d}{dt} \int \psi  (\partial_x v_1) v_2  & = \int \psi  (\partial_x \dot v_1)  v_2 +  \int \psi  (\partial_x v_1) \dot v_2 \\
& =  \int \psi  (\partial_x v_2) v_2 +  \int \psi  (\partial_x v_1) (\partial_x^2 v_1-2v_1+3(1-H^2)v_1) -   \alpha \int \psi(\partial_x v_1) f   \\
&\qquad + \int \psi \left((\partial_x F_1)  v_2   +   (\partial_x v_1) F_{2}\right)\\
& = -\frac 12 \int \psi' \left( v_2^2 + (\partial_x v_1)^2 - 2v_1^2\right) - \frac32 \int  \left(\psi (1-H^2)\right)' v_1^2
+  \alpha  \int  v_1  (\psi f)'   \\ 
& \qquad + \int \psi \left((\partial_x F_1)  v_2   +   (\partial_xv_1) F_{2}\right).
\end{aligned}
\end{equation}
Second, 
\begin{equation}\label{Itwo}
\begin{aligned}
 \frac{d}{dt} \int \psi' v_1v_2 & =   \int \psi' \dot v_1 v_2 +  \int \psi'  v_1 \dot v_2 \\
& =  \int \psi'  v_2^2 + \int \psi' v_1 \left(\partial_x^2 v_1 -2v_1+3(1-H^2)v_1\right) -   \alpha \int \psi' v_1 f
 + \int \psi' \left( F_1 v_2 +  v_1 F_2\right)\\
& =  \int \psi'\left(v_2^2 -(\partial_x v_1)^2 -2 v_1^2\right) + \frac 12 \int \psi''' v_1^2 
+ 3 \int \psi' (1-H^2)v_1^2 -   \alpha \int \psi'v_1 f \\ & + \int \psi' \left( F_1 v_2 +  v_1 F_2\right).
\end{aligned}
\end{equation}
Therefore,
\begin{equation}
\label{moondawn}
\begin{aligned}
\frac d{dt} {\mathcal I} &=- \mathcal B(v_1) +  \alpha \int v_1\left(\psi f'+\frac 12 \psi' f\right)  
+ \int  v_2  \left(\psi \partial_x F_1 +\frac 12 \psi' F_1\right) 
- \int v_1 \left( \psi \partial_x F_2 +\frac 12 \psi' F_2\right),
\end{aligned}
\end{equation}
where
\be\label{B_0}
\mathcal B(v_1):= \int \psi' (\partial_x v_1)^2 -\frac{1}{4}\int\psi'''v_1^2 - 3 \int \psi  HH' v_1^2 .
\ee
For a smooth function $g$  to be chosen later, let
\begin{equation}\label{defJ}
\mathcal  J :=\alpha  \int v_2 g - 2\mu \beta \int v_1 g,
\end{equation}
Then, using \eqref{syd}
\begin{equation}\label{moondawn 2}\begin{aligned}
\frac d{dt}{\mathcal J} & = \dot \alpha \int v_2 g + \alpha \int \dot v_2 g - 2 \mu \dot \beta \int v_1 g - 2 \mu \beta \int \dot v_1 g\\
& = \alpha \int \left( - \mathcal L g+ 4 \mu^2 g\right) v_1 -   \alpha^2 \int fg \\
& \qquad + F_\alpha \int v_2 g - 2 \mu F_\beta \int v_1 g + \alpha \int F_2 g - 2 \mu \beta \int F_1 g.
\end{aligned}\end{equation}

Adding (\ref{moondawn}) and (\ref{moondawn 2}), we obtain
\begin{equation}
\label{IpJ}
\frac{d}{dt}\left(\mathcal I + \mathcal J\right)=- \mathcal D(v_1,\alpha)+\mathcal R_{\mathcal D},
\end{equation}
where 
\begin{equation}
\label{def B}
\mathcal D(v_1,\alpha) :=\mathcal B(v_1)-\alpha \int v_1 \left( \psi f'+\frac 12\psi' f-\mathcal L g+4\mu^2 g\right)+ \alpha^2\int fg,
\end{equation}
and the rest term is
\begin{equation}
\label{virial error}
\mathcal R_{\mathcal D} :=
 \int  g\left(  \alpha F_2   - 2 \mu \beta   F_1 \right)
+ \int  v_2  \left(\psi \partial_x F_1 +\frac 12 \psi' F_1 + g F_\alpha  \right) 
- \int v_1 \left( \psi \partial_x F_2 +\frac 12 \psi' F_2 + 2 \mu g F_\beta  \right) .
\end{equation}

\subsection{Coercivity of the bilinear form $\mathcal B$}
Now we choose a specific function $\psi$ and we consider the question of the coercivity of the bilinear form $\mathcal B$ given in \eqref{B_0}. 
Let $\lambda>1$ to be chosen (we anticipate that in the sequel we set $\lambda=8$) and in the definition of $\mathcal I$ let 
\begin{equation}\label{defpsi}
\psi(x) : = \lambda\sqrt{2} H\Big(\frac{x}{\lambda}\Big)= \lambda\sqrt{2}\tanh\Big(\frac{x}{\lambda \sqrt{2}}\Big);
\quad \zeta(x) :=\sqrt{\psi'(x)}=\sech \Big(\frac{x}{\lambda\sqrt{2}}\Big).
\end{equation}
Note that $\zeta>0$ everywhere. Let $w$ be the following auxiliary function
\be\label{def w}
w:=\zeta v_1.
\ee
First, note that by integration by parts,
\begin{align*}
\int w_x^2 
&= \int (\zeta \partial_x v_1 +\zeta' v_1  )^2
  = \int \psi' (\partial_x v_1)^2 + 2 \int \zeta \zeta' v_1 (\partial_x v_1) + \int (\zeta')^2 v_1^2 \\
 & = \int \psi' (\partial_x v_1)^2  - \int \zeta \zeta'' v_1^2 \\
 & = \int \psi' (\partial_x v_1)^2  - \int\frac{\zeta''}{\zeta} w^2.
\end{align*}
Thus,
\begin{equation}\label{idgrad}
\int \psi' (\partial_x v_1)^2 = \int w_x^2  +  \int\frac{\zeta''}{\zeta} w^2.
\end{equation}
Second,
$$
\int\psi'''v_1^2 = \int \frac{(\zeta^2)''}{\zeta^2} w^2
= 2 \int \left(\frac {\zeta''}{\zeta}+ \frac{(\zeta')^2}{\zeta^2} \right) w^2.
$$
Therefore,
\begin{equation}\label{change}\begin{aligned}
\mathcal B(v_1) &= \int \psi' (\partial_x v_1)^2 -\frac{1}{4}\int\psi'''v_1^2 - 3 \int \psi  HH' v_1^2 \\
& = \int w_x^2 + \frac 12  \int \left( \frac {\zeta''}{\zeta} -\frac{(\zeta')^2}{\zeta^2} \right) w^2
- 3 \int \frac{\psi}{\psi'} HH' w^2.
\end{aligned}\end{equation}
Set
\[
{\mathcal B}^\sharp(w):= \int \left(w_x^2 - V w^2\right),\quad \hbox{where} \quad
V:= - \frac 12    \left( \frac {\zeta''}{\zeta} -\frac{(\zeta')^2}{\zeta^2} \right)+ 3  \frac{\psi}{\psi'} HH',
\]
so that 
\be\label{B sharp}
 {\mathcal B}^\sharp(w)=\mathcal B(v_1).
\ee
Note that by \eqref{defpsi} and direct computations,
\[
\frac{\zeta''}{\zeta}= \frac1{2\lambda^2} (1-2\zeta^2),\quad
\frac{(\zeta')^2}{\zeta^2}= \frac1{2\lambda^2} (1-\zeta^2),
\quad
\frac {\zeta''}{\zeta}-\frac{(\zeta')^2}{\zeta^2}=-\frac{\zeta^2}{2\lambda^2}
=-\frac{1}{2\lambda^2}\sech^2\big(\frac{x}{\lambda\sqrt{2}}\big).
\]
Moreover,
\[
\frac{\psi}{\psi'} = \lambda\sqrt{2}\tanh\Big(\frac{x}{\lambda\sqrt{2}}\Big)\cosh^2\Big(\frac{x}{\lambda\sqrt{2}}\Big).
\]
Therefore, ${\mathcal B}^\sharp(w)= \int \left(w_x^2 - V w^2\right)$, where
\be\label{V}
V(x)=
\frac{1}{4\lambda^2}\sech^2\Big(\frac{x}{\lambda\sqrt{2}}\Big)
+3\lambda\tanh\Big(\frac{x}{\lambda\sqrt{2}}\Big)\cosh^2\Big(\frac{x}{\lambda\sqrt{2}}\Big)\tanh\Big(\frac{x}{\sqrt{2}}\Big)\sech^2\Big(\frac{x}{\sqrt{2}}\Big).
\ee
Recall $Y_1$ from \eqref{phi_1}. Let 
\be\label{Z_1}
Z_1 :=Y_1\cosh\Big(\frac{x}{\lambda\sqrt{2}}\Big) \quad \hbox{so that}\quad
\langle v_1,Y_1\rangle=0 \ \iff \ \langle w,Z_1\rangle=0.
\ee
Note that $Z_1$ is odd. In the following, we claim the following coercivity result for the  quadratic form $\mathcal B^\sharp$.

\begin{lemma}\label{le:posVir}
Let $\lambda = 8$. There exists $\kappa>0$ such that,
for any odd function $w\in H^1$, 
\begin{equation}\label{posVir1}
 \langle w,Z_1 \rangle   =0\quad
  \implies \quad
    {\mathcal B}^\sharp(w)
  \geq \kappa \int w_x^2.
\end{equation}
\end{lemma}

\begin{proof} 
First, note that 
\begin{equation}
\mathcal B^\sharp(w)=\frac 14 \mathcal B^\sharp_1(w) + \mathcal B^\sharp_2(w), \quad \hbox{where}
\label{decomp pos Vir}
\end{equation}
where
\[
\mathcal B^\sharp_1(w) =\int   w_x^2 - \frac 1{\lambda^2} \sech^2\left(\frac x{\lambda \sqrt{2}}\right);
\]
\[
\mathcal B^\sharp_2(w) = \int  \left(\frac{3}{4} w_x^2-V_2 w^2\right),\quad 
V_2=3\lambda\tanh\Big(\frac{x}{\lambda\sqrt{2}}\Big)\cosh^2\Big(\frac{x}{\lambda\sqrt{2}}\Big)\tanh\Big(\frac{x}{\sqrt{2}}\Big)\sech^2\Big(\frac{x}{\sqrt{2}}\Big).
\]
The idea is to show that $\mathcal B^\sharp_1(w)$ and $\mathcal B^\sharp_2(w)$ are both non negative.

\medskip

{\bf Step 1.} We claim
\begin{claim}\label{cc1}
For all $\lambda_0 >0$, for any odd function $w\in H^1$, 
\begin{equation}\label{cc2}
 \int w_x^2 - \frac{1}{\lambda_0^2}\sech^2\Big(\frac{x}{\lambda_0\sqrt{2}}\Big) w^2 \geq 0.
\end{equation}
\end{claim}
\begin{proof}[Proof of Claim \ref{cc1}]
By a change of variables, we can restrict ourselves to the case $\lambda_0\sqrt{2}=1$.
Note that the classical operator $-\partial_x^{2} - 2 \sech^2(x)$ has a continuous spectrum $[0,+\infty)$ and only one negative eigenvalue $-1$, associated to the even eigenfunction $\sech(x)$ (see Titchmarsh \cite[\S 4.18]{Tit}). Thus, by the spectral theorem, for any odd function  $w\in H^1$,
\[
 \int w_x^2 - 2\sech^2(x) w^2\geq 0.
\]
\end{proof}
\noindent
The rest of the proof is devoted to showing that the second term is nonnegative as well. 

\medskip

{\bf Step 2.} Let $\lambda=8$. We observe the following elementary inequality
\begin{equation}\label{eq:maj}
 \forall x\in \R,\quad
 V_2(x)=  3\lambda\tanh\Big(\frac{x}{\lambda\sqrt{2}}\Big)\cosh^2\Big(\frac{x}{\lambda\sqrt{2}}\Big)\tanh\Big(\frac{x}{\sqrt{2}}\Big)\sech^2\Big(\frac{x}{\sqrt{2}}\Big) 
  < \frac{21}{10}\sech^2\Big(\frac{x}{2}\Big).
\end{equation}
(This is easily checked by numerics.)  In particular, combining \eqref{cc2} with $\la_0=8$ and \eqref{eq:maj} in \eqref{decomp pos Vir}, we have, for any odd function $w$,
\begin{equation}\label{majoration}
\mathcal B^\sharp_2 (w) \geq \int \left(\frac 34  w_x^2 - \frac{21}{10}\sech^2\Big(\frac{x}{2}\Big)  w^2\right).
\end{equation}

\medskip

{\bf Step 3.} Finally, we claim the following coercivity property: there exists $\kappa>0$ such that
\begin{equation}\label{posVir3}
\langle w,Z_1\rangle  =0\quad
\implies \quad
\int \left(\frac 34w_x^2 - {\frac{21}{10}\sech^2\Big(\frac{x}{2}\Big)}  w^2\right)
\geq \kappa \int w_x^2.
\end{equation}
Note that in view of \eqref{majoration},  this statement completes the proof of the lemma.
\medskip

\noindent
{{\it Proof of  \eqref{posVir3}}}.
Let 
\be\label{y_1a}
{\tilde Y_1}(x) := \sqrt{\frac{15}8} \frac{\sinh\left(\frac{x}{2}\right)}{\cosh^3\left(\frac{x}{2}\right)}, \quad 
\langle {\tilde Y_1},{\tilde Y_1} \rangle =1.
\ee
We claim that for any odd function $w$,
\begin{equation}\label{posVir4}
\int \left(w_x^2 -   {3}\sech^2\left(\frac{x}{2}\right)   w^2\right) +\langle {\tilde Y_1}, w\rangle^2  
\geq 0.
\end{equation}
Indeed, the operator 
$
 - \partial_x^2 - {3} \sech^2\left(\frac{x}{2}\right)
$
is a classical operator which has exactly three negative eigenvalues (see Titchmarsh \cite[\S 4.18]{Tit}) 
$
-\frac 94,$ $ -1,$ and $-\frac 14,$ and continuous spectrum $[0,+\infty)$.
The eigenvalues $-\frac 94$ and $-\frac 14$ are associated to the even eigenfunctions 
\[
\sech^{3}\Big(\frac x2 \Big) \quad \hbox{ and } \quad  \sech^{3}\Big(\frac x2 \Big)-\frac 45 \sech \Big(\frac x2\Big),
\]
respectively; on the other hand, $-1$ is associated to the odd eigenfunction ${\tilde Y_1}$ defined in \eqref{y_1a}.

\medskip
For $w$ odd, we use the following (orthogonal) decomposition:  $w = w_1 + \langle   {\tilde Y_1},w\rangle {\tilde Y_1}$, where $\langle w_1, {\tilde Y_1}\rangle=0$.
Then, since ${\tilde Y_1}$ is an eigenfunction of $- \partial_x^2 - {3} \sech^2\left(\frac{x}{2}\right)$ with eigenvalue $-1$, we obtain
\[
\int w_x^2 -  3 \sech^2\left(\frac{x}{2}\right)  w^2
= \int (w_1)_x^2 - 3 \sech^2\left(\frac{x}{2}\right)  w_1^2 - \langle {\tilde Y_1} , w\rangle^2.
\]
Additionally, since $w_1$ is odd and orthogonal to ${\tilde Y_1}$, by the spectral theorem, we obtain
\[
\int (w_1)_x^2 -  3 \sech^2\left(\frac{x}{2}\right)  w_1^2\geq 0,
\]
and \eqref{posVir4} follows.

\medskip

Let us come back to the proof of \eqref{posVir3}. We note that, using \eqref{posVir4},
\begin{align*}
\int \left(\frac 34 w_x^2 -  \frac{21}{10}\sech^2\left(\frac{x}{2}\right)  w^2\right)
& = \frac {1}{20} \int w_x^2 + \frac 7 {10} \int \Big(w_x^2 -   {3}\sech^2\left(\frac{x}{2}\right)  w^2\Big)
\\
& \geq  \frac {1}{20}\left(  \int w_x^2   - 14 \langle {\tilde Y_1}, w\rangle^2\right).
\end{align*}
Now, we estimate $\langle {\tilde Y_1}, w\rangle^2$ using the orthogonality condition $ \langle  w,Z_1\rangle=0$. Let $\nu\in \R$ and
\[
\xi_\nu := {\tilde Y_1} - \nu Z_1 .
\]
Since ${\tilde Y_1}$ and $Z_1$ are odd, there exists an even function $\upsilon_\nu\in \mathcal S(\RR)$ such that $\upsilon_\nu ' = \xi_\nu $.
In particular,
\[
 \langle {\tilde Y_1}, w \rangle= \langle  \xi_\nu, w \rangle= - \langle \upsilon_\nu, w_x\rangle,
\]
and using Cauchy-Schwarz' inequality,
\[
\langle {\tilde Y_1}, w\rangle^2  \leq
\left(\int w_x^2\right) \left(   \min_{\nu\in \RR}  \int \upsilon_\nu^2  \right) .
\]
Note that $\tilde Y_1$ in \eqref{y_1a} and $Z_1$ in \eqref{Z_1} are explicit functions, so that $\xi_\nu$ and $\upsilon_\nu$ are easily computable. A numerical computation gives that
\[
   \min_{\nu\in \RR}\int \upsilon_\nu^2  \approx  0.04 < \frac 1{14} \approx 0.071.
\]
Thus, \eqref{posVir3} and \eqref{posVir1} are proved.
\end{proof}

\begin{remark}\label{alter}
In this proof, we have used numerical computations of elementary integrals.
We present briefly an alternative proof of the fact that \eqref{posVir4} implies \eqref{posVir3} relying more strongly on numerics.
Let
\be\label{L_sharp}
  {\mathcal L}^\sharp:= -\partial_x^2 -V ,
\ee
be the linear operator representing $\mathcal B^\sharp$. From  the arguments of the proof in \cite[Lemma 11]{MR2150386}, there exists
a function $Z_1^\sharp\in L^\infty$  such that $\int ((Z_1^\sharp)_x)^2<\infty$ and $  {\mathcal L}^\sharp Z_1^\sharp=Z_1$.
Using numerical computations, we observe
\begin{equation}\label{YZ}
\langle {\mathcal L}^\sharp Z_1^\sharp , Z_1^\sharp \rangle = \langle Z_1 , Z_1^\sharp \rangle \approx - 2.63<0.
\end{equation}
We find the function $Z_1^\sharp$  by a shooting method  (in particular, we obtain $Z_1'(0)\approx -0.4376$), and then 
$\langle Z_1 , Z_1^\sharp \rangle$ by numerical integration.
Using the arguments of the proof of Lemma 13 in \cite{MR2150386}, \eqref{YZ} and \eqref{posVir4} imply \eqref{posVir3}.
\end{remark}
\begin{remark}
For the  case of general perturbations i.e. with the oddness restriction (see Remark \ref{re:general}), one would need to show coercivity of the bilinear form $\mathcal B^\sharp$ under the assumptions of orthogonality with respect to $Z_0$ and $Z_1$,  
where $Z_0 =Y_0\cosh\Big(\frac{x}{\lambda\sqrt{2}}\Big)$. This property is definitely more delicate to prove since the bilinear form appearing in (\ref{posVir3}) will not be sufficient as a lower bound for $\mathcal B^\sharp(w)$. Indeed, one can show that the index of the associated linear operator $- \partial_x^2 - {3} \sech^2\left(\frac{x}{2}\right)$ is $3$. However, we have numerical evidence that the index of the general operator ${\mathcal L}^\sharp$ (see \eqref{L_sharp}) is $2$, as it should be. 
\end{remark}

\begin{corollary}\label{cor:4.1}
Let $\lambda = 8$. There exists   $\kappa>0$ such that for 
any odd function  $w\in L^\infty$ with $w_x \in L^2$ it holds 
\begin{equation}\label{posVir1bis}
 \int  w Z_1   =0\quad
  \implies \quad
    {\mathcal B}^\sharp(w)
  \geq \kappa \int w_x^2.
\end{equation}
\end{corollary}
\begin{proof}
First note that a standard consequence of \eqref{posVir1} is the following estimate:
there exists a  constant $C>0$, depending on $\|Z_1\|_{H^1}$ and the constant in estimate  (\ref{posVir1}) of  Lemma \ref{le:posVir}, and a constant $\kappa>0$,   such that for any odd function 
$w\in H^1$, 
\begin{equation}\label{posVir1tri}
     {\mathcal B}^\sharp(w) + C|\langle w,Z_1 \rangle|^2  \geq \kappa \int w_x^2.
\end{equation}

Let $\chi$ be a  smooth even function such that $\chi(x)=1$ for $|x|<1$, $\chi(x)=0$ for $|x|>2$ and $0\leq \chi\leq 1$. For $A>1$, let $\chi_A(x) = \chi(x/A)$. 
Let $w\in L^\infty$, odd and such that $w_x \in L^2$, $\langle w,Z_1 \rangle = 0$. Set $w_A = w \chi_A$.
Then, we claim
\begin{equation}\label{AAA}
\|w_x\|_{L^2}^2 - \|\partial_x w_A\|_{L^2}^2 = o_A(1),\quad
\int V w^2 - \int V w_A^2 = o_A(1),\quad \langle w_A,Z_1 \rangle = o_A(1),
\end{equation}
where $o_A(1)$ denotes a function such that $\lim_{A\to +\infty} o_A(1)=0$ (possibly depending on $w$).
Indeed,  from direct computations
\[
\left| \int w_x^2 - \int (\partial_x w_A)^2 \right|
\leq \left| \int w_x ^2 (1-\chi_A^2)\right|  + \left| \int w^2 \chi_A'' \chi_A \right|
\lesssim \int_{|x|>A} w_x^2 + \frac 1 A \|w\|_{L^\infty}^2;
\]
\[
\left| \int V w^2 - \int V w_A^2 \right| 
\leq \|w\|_{L^\infty}^2 \int_{|x|>A} V;
\]
\[
\left| \langle w_A,Z_1 \rangle\right|  \leq \left|\int  wZ_1 \right| + \left| \int Z_1 (w-w_A)\right| 
\leq \left| \int Z_1 (w-w_A)\right|\leq  \|w\|_{L^\infty} \int_{|x|>A} |Z_1|.
\]

Applying \eqref{posVir1tri} to $w_A \in H^1$ (which is odd), we obtain
\[
{\mathcal B}^\sharp(w_A) + C |\langle w_A,Z_1 \rangle|^2  \geq \kappa \int (w_A)_x^2.
\]
Thus, by \eqref{AAA},
\[
{\mathcal B}^\sharp(w)   \geq \kappa \int w_x^2 + o_A(1).
\]
and  we obtain the result passing to the limit  $A\to +\infty$.
\end{proof}

\subsection{Coercivity of the bilinear form $\mathcal D$}\label{coer D}
Let us come back to the modified quadratic form $\mathcal D$ introduced in \eqref{def B}. Written in terms of $w$ (see \eqref{def w} and \eqref{B sharp}), we have
\begin{equation}
\label{def Bbis}
\mathcal D(v_1,\alpha)=\mathcal D^\sharp(w,\alpha) :=\mathcal B^\sharp(w)-\alpha \int w \frac 1{\zeta}\left(\psi f'+\frac 12\psi' f-\mathcal L g+4\mu^2 g\right) + \alpha^2\int fg.
\end{equation}
Now we use a modified version of the Fermi Golden Rule to choose a particular $g$. Let $k$ be the function defined in (\ref{var par 5}). Let  
\be\label{def a}
a:=-\frac{\langle\psi f'+\frac 12\psi' f , {\rm Im} (k)\rangle}{\langle \psi'f , {\rm Im} (k)\rangle},
\quad \hbox{so that}\quad
\left\langle \psi f'+\left(a+\frac 12\right)\psi' f, {\rm Im} (k)\right\rangle = 0.
\ee
Numerically, $a\approx   0.687271$. From condition \eqref{fermi 2} in Remark \ref{FGrule}, which ensures that $\langle \psi'f , {\rm Im} (k)\rangle\neq 0$ ($\langle \psi'f , {\rm Im} (k)\rangle \approx -0.327$), and  Lemma \ref{le:ODE}, there exists a unique real Schwartz solution $g$  of 
\be\label{def g}
\mathcal L g -4\mu^2 g= \psi f'+\left(a+\frac 12\right)\psi' f .
\ee
Moreover, in view of the decay of $Y_1$ \eqref{phi_1}, the decay of $f$ \eqref{deff}, and the explicit formula in \eqref{var par 7}, the function $g$ satisfies
\begin{equation}\label{decaygg}
\forall x\in \RR,\quad 
|g(x)|+|g'(x)|\lesssim e^{-\frac {|x|}{\sqrt{2}}}.
\end{equation}
Consequently, \eqref{def g} leads to the simplified expression
\[
\mathcal D^\sharp(w,\alpha)=\mathcal B^\sharp(w)+\alpha  a\int w  \zeta f + \alpha^2\int fg.
\]
Recall that from \eqref{posVir1} we already know that, under $\langle w, Z_1\rangle =0,$ $\mathcal B^\sharp(w) $ is bounded below by $\int w_x^2$. Numerical computations, using \eqref{var par 7}, show that 
\be\label{fg}
 \langle f, g\rangle \approx 0.0163 >0.
\ee
Now we prove
\begin{lemma}\label{B coerc}
Let $\lambda=8$.
There exists $\kappa>0$  such that for any odd $H^1$ function $w$ satisfying $ \langle w, Z_1\rangle=0$,
\begin{equation}
\label{B coerc 1}
\mathcal D^\sharp(w,\alpha) \geq \kappa \left(\alpha^2 + \int  w_x^2 \right).
\end{equation}
\end{lemma}
\begin{proof}
Denote $h:=a\zeta f+b Z_1$, where  $b$ is chosen so that  (see \eqref{YZ})
\[
\langle h , Z_1^\sharp\rangle=0.
\] 
Here  $Z_1^\sharp$ solves $\mathcal L^{\sharp} Z_1^\sharp=Z_1$ and is such that $Z_1^\sharp\in L^\infty$, $(Z_{1}^\sharp)_x\in L^2$  (see Remark \ref{alter}).
Note that $h$ is odd. Let $h^\sharp\in L^\infty$ be the odd function such that  $\mathcal L^\sharp h^\sharp=h$ and $\int (h_x^\sharp)^2<\infty$.
Since $\langle w, Z_1\rangle=0$, we have 
$
a \int w\zeta f= \langle w , h\rangle 
$, and thus
\[
\mathcal D^\sharp(w,\alpha)=\mathcal B^\sharp(w)+\alpha \langle w ,h\rangle  + \alpha^2 \langle f, g\rangle.
\]
Furthermore we observe
\[
\langle h^\sharp, Z_1\rangle = \langle h^\sharp, \mathcal L^\sharp Z_1^\sharp\rangle = \langle h, Z_1^\sharp \rangle =0,
\]
hence (since $h^\sharp$ is odd), by Corollary \ref{cor:4.1},
\[
0<\mathcal B^\sharp(h^\sharp)= \langle \mathcal L^\sharp h^\sharp, h^\sharp \rangle = \langle h^\sharp,h \rangle.
\]
Given $w$ odd, we decompose as follows
\be\label{deco1}
w=w^\perp+ c \, h^\sharp  \quad \mbox{where}\quad \langle w^\perp,  h\rangle =\langle w^\perp , \mathcal L^\sharp h^\sharp \rangle=0
\quad \hbox{and} \quad \langle w^\perp,Z_1\rangle = \langle w,Z_1\rangle =0.
\ee
Thus, by orthogonality
\[
\mathcal B^\sharp(w) =\mathcal B^\sharp(w^\perp)+ c^2\mathcal B^\sharp(h^\sharp),
\]
and
\[
\mathcal D^\sharp(w, \alpha)= \mathcal B^\sharp(w^\perp)+\left( c^2  +c \alpha \right) \langle h^\sharp,h\rangle +\alpha^2 \langle f, g\rangle.
\]
Numerical computations show  that
\begin{equation}
\label{numer 2}
\langle h^\sharp,h\rangle\approx 0.0147. 
\end{equation}
In particular, since $0<\langle h^\sharp,h\rangle  < 4 \langle f, g\rangle $ (see \eqref{fg}),  we obtain for all $\alpha\in \RR$, $c\in \RR$,
\[
\left( c^2  +c \alpha \right) \langle h^\sharp,h\rangle +\alpha^2 \langle f, g\rangle \gtrsim \alpha^2 + c^2.
\]
Combined with Corollary \ref{cor:4.1} and \eqref{deco1}, this gives
\[
\mathcal D^\sharp(w, \alpha) \gtrsim \alpha^2 + c^2 + \int (w_x^\perp)^2
\gtrsim \alpha^2  + \int w_x^2,
\]
as required.
\end{proof}

\begin{remark}\label{shooting}
The values in \eqref{numer 2} are obtained using elementary numerical computations. We  find the functions $h^\sharp$ and $g$ by  shooting  
(in particular, we find $(h^\sharp)'(0)\approx 0.0249$ and $g'(0)\approx -0.333$) and then use numerical integration. Observe that   the functions considered here or their derivatives are exponentially decaying, which facilitates the computations.
\end{remark}

\section{End of the proof of Theorem \ref{TH1}}\label{ERROR}

As in statement of Theorem \ref{TH1}, we consider an odd function $\varphi^{in}\in H^1\times L^2$ satisfying $\|\varphi^{in}\|_{H^1\times L^2}<\varepsilon$ for some small $\varepsilon>0$ to be chosen. By Proposition \ref{henry}, the corresponding solution $\varphi(t)$ of \eqref{eqvarphi} is global in $H^1\times L^2$ and satisfies, for all $t\in \R$,
\begin{equation}\label{smallphi}
 \|\varphi(t)\|_{H^1\times L^2} \lesssim \varepsilon.
\end{equation}
Now we use the decomposition and computations of Section \ref{sec:2}, introducing in particular the functions $u(t)$, $z(t)$, $\alpha(t)$, $\beta(t)$ and $v(t)$
as in \eqref{defz}, \eqref{defu}, \eqref{defzu}, \eqref{defab} and \eqref{defv}. 
Note that from \eqref{smallphi}, it holds
\begin{equation}\label{allsmall}
\forall t\in \RR,\quad \|u(t)\|_{H^1\times L^2} +\|v(t)\|_{H^1\times L^2}+\|u_1(t)\|_{L^\infty} +\|v_1(t)\|_{L^\infty}+
|z(t)|\lesssim \varepsilon.
\end{equation}
In order to simplify some estimates, we define
\begin{equation}\label{nloc}
\|v_1\|_{H^1_{\omega}}^2 :=
\int \left( |\partial_x v_1|^2 + v_1^2   \right) \sech\left(\frac  x{2 \sqrt{2}}\right),\quad
\|v_2\|_{L^2_{\omega}}^2 :=
\int  v_2^2   \sech\left(\frac  x{2 \sqrt{2}}\right),
\end{equation}
and
\begin{equation}\label{nloc2}
\|v\|_{H^1_{\omega}\times L^2_{\omega}}^2 :=\|v_1\|_{H^1_{\omega}}^2 + \|v_2\|_{L^2_{\omega}}^2.
\end{equation}

\subsection{Control of the error terms and conclusion of the Virial argument}
The key ingredient of the proof of asymptotic stability in the energy space is the following result.
\begin{proposition}\label{pr:11}
For $\varepsilon>0$ small enough,
\begin{equation}\label{11un}
\int_{-\infty}^{+\infty} \left( |z(t)|^4  + \|v(t)\|_{H^1_{\omega}\times L^2_{\omega}}^2  \right) dt 
\lesssim \varepsilon^2 .
\end{equation}
\end{proposition}

\begin{proof}
Let 
\be\label{ga}
\gamma(t) := \alpha(t)\beta(t),
\ee
and recall the virial-type quantities $\mathcal I(t)$, $\mathcal J(t)$, already defined in \eqref{defI}, \eqref{defJ} for
$\psi$ as in \eqref{defpsi} for $\lambda=8$ and $g$ introduced in \eqref{def g}.

\medskip

The proof of \eqref{11un} is based on a suitable combination of the following three estimates, which hold for some fixed constants $\kappa_0, C>0$:
\begin{align}
 \frac d{dt} \gamma &\geq   2 \mu ( \beta^2- \alpha^2 )  - C \varepsilon \left( |z(t)|^4+  \|v_1\|_{H^1_{\omega}}^2\right),
\label{10deux}\\
 -\frac d{dt} (\mathcal I + \mathcal J)
&\geq   \kappa_0  \left( \alpha^2+ \|v_1\|_{H^1_{\omega}}^2   \right)
- C \varepsilon   \left( |z(t)|^4+\|v_2\|_{L^2_{\omega}}^2\right) , 
\label{10un} \\
 2 \frac d{dt} \int \sech\left(\frac x{2 \sqrt{2}}\right) v_1 v_2 
&\geq    \|v_2\|_{L^2_{\omega}}^2
-C\left(  |z(t)|^4+    \|v_1\|_{H^1_{\omega}}^2 \right).
\label{10trois}
\end{align}

\medskip

{\bf Step 1.} Proof of  \eqref{11un} assuming \eqref{10deux}, \eqref{10un}   and \eqref{10trois}.
For $\sigma>0$ small to be chosen, let
\[
\mathcal K :=   \frac{\kappa_0}{4\mu} \gamma -(\mathcal I + \mathcal J) +  2\sigma \int \sech\left(\frac x{2 \sqrt{2}}\right) v_1 v_2.
\]
Recall that $\mathcal K $ is a time dependent function. From  \eqref{10deux}, \eqref{10un} and \eqref{10trois}, we obtain
\[
\frac d{dt}{\mathcal K} \geq \frac {\kappa_0} 2 (\alpha^2 + \beta^2) + \kappa_0 \|v_1\|_{H^1_{\omega}}^2  +\sigma \|v_2\|_{L^2_{\omega}}^2 - C \left(\sigma+\varepsilon\right) \left( |z(t)|^4+ \|v_1\|_{H^1_{\omega}}^2\right) - C \varepsilon \|v_2\|_{L^2_{\omega}}^2 .
\]
Note that from \eqref{defab} we have $\alpha^2+\beta^2 = |z|^4$. Thus,  choosing $\sigma>0$ sufficiently small, and then $\varepsilon>0$  small enough, we obtain
\begin{equation}\label{dtK}
\frac d{dt}{\mathcal K} \gtrsim   |z(t)|^4 +   \|v\|_{H^1_{\omega}\times L^2_{\omega}}^2  .
\end{equation}
By the expressions of $\mathcal I$, $\mathcal J$, $\gamma$ and \eqref{allsmall}, we easily  see that
\begin{equation}\label{Kt}
\forall t\in \RR, \quad |\mathcal K(t)|\lesssim \|v(t)\|_{H^1\times L^2}^2 + |z(t)|^4 \lesssim \varepsilon^2.
\end{equation}
Therefore, integrating \eqref{dtK} on $[-t_0,t_0]$ and passing to the limit as $t_0\to +\infty$, we find \eqref{pr:11}.

\medskip

To finish the proof, we only have to prove \eqref{10deux}, \eqref{10un} and \eqref{10trois}.

\medskip

\textbf{Step 2.} Preliminary computations and estimates.
First,  from the computations of Section \ref{sec:2}, we give the expressions of $F_\alpha$, $F_\beta$, $F_1$ and $F_2$ in \eqref{syd}.
Note that from \eqref{defz}, \eqref{eqvarphi} and the fact that $\mathcal L Y_1 = \mu^2 Y_1$,
\begin{equation}\label{5p12}\left\{\begin{aligned}
& \dot z_1 = \langle \dot \varphi_1,Y_1\rangle =  \langle  \varphi_2,Y_1\rangle = \mu z_2, \\
& \dot z_2 = \frac1\mu \langle \dot \varphi_2,Y_1\rangle =   -\frac1\mu \langle \mathcal L \varphi_1 +(3 H \varphi_1^2 + \varphi_1^3),Y_1\rangle    = - \mu z_1 -\frac 1{\mu}  \langle 3H\varphi_1^2 +\varphi_1^3,Y_1\rangle.
\end{aligned}\right.\end{equation}
(Recall that $\varphi_1=z_1 Y_1+u_1$ and $\langle u_1,Y_1\rangle =0$, see \eqref{OrthoU}.) 
In particular, \eqref{defab} leads to
\[
\dot \alpha = 2 z_1 \dot z_1 - 2 z_2 \dot z_2 = 2 \mu \beta + \frac 2\mu z_2 \langle 3 H \varphi_1^2 + \varphi_1^3, Y_1\rangle,
\]
\[
\dot \beta = 2 z_1 \dot z_2 + 2 \dot z_1 z_2 = - 2 \mu \alpha - \frac 2\mu z_1 \langle 3 H \varphi_1^2 + \varphi_1^3, Y_1\rangle.
\]
Thus, we obtain in \eqref{syd},
\[
F_\alpha := \frac 2 \mu z_2 \langle 3 H \varphi_1^2 + \varphi_1^3, Y_1\rangle,\quad
F_\beta := - \frac 2 \mu z_1 \langle 3 H \varphi_1^2 + \varphi_1^3, Y_1\rangle.
\]
Next, since
\begin{equation}\label{dz2}
\frac d{dt} |z|^2 = 2  z_1 \dot z_1 + 2 \dot z_2 z_2 =   -\frac 2\mu z_2 \langle 3 H \varphi_1^2 + \varphi_1^3, Y_1\rangle
= -F_\alpha,
\end{equation}
we deduce, from \eqref{defv}, \eqref{defu}, \eqref{eqvarphi} and \eqref{5p12}, that
\[
\begin{aligned}
\dot v_1 & = \dot u_1 + \frac d{dt} |z|^2  q = \dot\varphi_1 -\dot z_1 Y_1 - q F_\alpha\\
 &  =\varphi_2 - \mu z_2 Y_1 -q F_\al = u_2-q F_\al \\
 & =v_2 -qF_\al,
\end{aligned}
\]
and that in \eqref{eqofv} and \eqref{syd}, 
\begin{equation}
\label{f1 falph}
F_1 := - q F_\alpha.
\end{equation}
We have by direct computations
\[
\begin{aligned}
\dot u_2 & = \dot \varphi_2 -\mu \dot z_2 Y_1 = - \mathcal L \varphi_1 - (3 H \varphi_1^2 + \varphi_1^3)+ \mu^2 z_1Y_1 +   \langle 3H\varphi_1^2 +\varphi_1^3,Y_1\rangle Y_1 \\
& =- \mathcal L u_1  - (3 H (u_1 +z_1 Y_1)^2 +\varphi_1^3) +   \langle 3H(u_1 +z_1 Y_1)^2 + \varphi_1^3,Y_1\rangle Y_1\\
& =- \mathcal L u_1  - 3z_1^2 ( H Y_1^2 -  \langle H Y_1^2,Y_1\rangle Y_1 )    - (3 H (u_1^2 + 2z_1 u_1 Y_1) +\varphi_1^3) +   \langle 3H(u_1^2 + 2z_1 u_1 Y_1) + \varphi_1^3,Y_1\rangle Y_1\\
& = - \mathcal L u_1 -2 z_1^2 f +F_u,
\end{aligned}
\]
so that  in \eqref{eqofu},
\[
F_u := - \left[3 H(u_1^2 + 2u_1 z_1 Y_1)+\varphi_1^3  - 
\langle 3 H(u_1^2 + 2u_1 z_1 Y_1)+\varphi_1^3 , Y_1 \rangle Y_1 \right].
\]
We also observe that 
\[
\dot v_2 = \dot u_2 
= - \mathcal L v_1 - \alpha f + F_u,
\]
and thus
\[
F_2 = F_u = - \left[3 H(u_1^2 + 2u_1 z_1 Y_1)+\varphi_1^3  - 
\langle 3 H(u_1^2 + 2u_1 z_1 Y_1)+\varphi_1^3 , Y_1 \rangle Y_1 \right].
\]

\medskip

Second, we prove
\begin{equation}\label{rwv}
\|v_1\|_{H^1_{\omega}} \lesssim  \|\partial_x w\|_{L^2},
\end{equation}
where we recall from Section \ref{VIRIAL} the notation $w(t,x) = v_1(t,x) \zeta(x)=v_1(t,x) \sech\left(\frac x{8\sqrt{2}}\right)$ (see \eqref{def w}).

Indeed, we  observe   that from \eqref{cc2} (with the choice $\lambda_0=100$),
\begin{equation}\label{h1}
\|\partial_x w\|_{L^2}^2\gtrsim
\int  \sech^2\left(\frac {x}{100}\right) w^2=
\int \sech^2\left(\frac {x}{100}\right)\sech^2\left(\frac {x}{8 \sqrt{2}}\right) v_1^2\gtrsim
\int \sech\left(\frac {x}{2 \sqrt{2}}\right) v_1^2 .
\end{equation}
Next, we have
\begin{align*}
\|\partial_x w\|_{L^2}^2 &\gtrsim \int \sech^2\left(\frac {x}{100}\right) |\partial_x w|^2 
  = \int \sech^2\left(\frac {x}{100}\right) |\zeta \partial_x v_1 + \zeta' v_1 |^2\\
& \gtrsim \int \sech^2\left(\frac {x}{100}\right) \sech^2\left(\frac {x}{8 \sqrt{2}}\right)|\partial_x v_1|^2
+ 2 \int \sech^2\left(\frac {x}{100}\right) \zeta \zeta' (\partial_x v_1) v_1 + \int \sech^2\left(\frac {x}{100}\right) (\zeta')^2 v_1^2\\
& \gtrsim \int \sech\left(\frac {x}{2 \sqrt{2}}\right)|\partial_x v_1|^2
+ \int v_1^2\left( - \left(\sech^2\left(\frac {x}{100}\right) \zeta \zeta'\right)' + \sech^2\left(\frac {x}{8 \sqrt{2}}\right) (\zeta')^2\right).
\end{align*}
Thus, using \eqref{h1},
\[
\int \sech\left(\frac {x}{2 \sqrt{2}}\right)|\partial_x v_1|^2
\lesssim \|\partial_x w\|_{L^2}^2+\int \sech^2\left(\frac {x}{100}\right) \sech^2\left(\frac {x}{8 \sqrt{2}}\right)v_1^2
\lesssim \|\partial_x w\|_{L^2}^2 ,
\]
which completes the proof of \eqref{rwv}.

\medskip

\textbf{Step 3.} Proof of \eqref{10deux}.
From  \eqref{syd} we find
\begin{equation}\label{moondawn 3}
\dot {\gamma} =
\dot \alpha \beta + \alpha \dot \beta 
= 2 \mu \left( \beta^2 - \alpha^2\right) + \mathcal R_{\gamma}, 
\quad \hbox{where} \quad \mathcal R_{\gamma} = \beta F_\alpha + \alpha F_\beta.
\end{equation}
Replacing $\varphi_1=u_1 + z_1 Y_1 = v_1 - |z|^2 q + z_1 Y_1$ in the expression of $F_\alpha$ and $F_\beta$,
then using the explicit decay of $Y_1$ (see \eqref{phi_1}) and \eqref{allsmall}, we have
\begin{equation}\label{Fab}
|F_\alpha|+|F_\beta|\lesssim |z| \left( |z|^2 + \|v_1\|_{L^2_{\omega}}^2\right).
\end{equation}
From the definition of $\alpha$, $\beta$ and \eqref{allsmall}  we obtain  finally 
\[
|\mathcal R_\gamma|\lesssim |z|^3 \left( |z|^2 + \|v_1\|_{L^2_{\omega}}^2\right)
 \lesssim \varepsilon \left(|z|^4 +  \|v_1\|_{L^2_{\omega}}^2\right).
\] 

\medskip
 
{\bf Step 4.} Proof of \eqref{10un}.
From \eqref{IpJ}, \eqref{B coerc 1} and \eqref{rwv}, it is sufficient to prove the following estimate
\begin{equation}\label{lll}
|\mathcal R_{\mathcal D}| \lesssim \varepsilon \left( |z(t)|^4+ \|\partial_x w\|_{L^2}^2 + \|v_2\|_{L^2_{\omega}}^2 \right),
\end{equation}
where $\mathcal R_{\mathcal D}$ is defined in \eqref{virial error}.

Using  \eqref{Fab},  \eqref{decaygg} and Cauchy-Schwarz inequality, we have
\begin{equation}\label{SFun}
\left|F_\alpha \int v_2 g\right| + \left|F_\beta \int v_1 g\right| \lesssim 
|z| \left( |z|^2 +  \|v_1\|_{L^2_{\omega}}^2\right) \left( \|v_1\|_{L^2_{\omega}}+\|v_2\|_{L^2_{\omega}}\right).
\end{equation}
From \eqref{decaygg}, (\ref{f1 falph}), (\ref{Fab}) and   \eqref{decayq}, we get
\begin{equation}\label{SFdeux}
\left|\beta \int g F_1 \right| + \left|\int v_2 \left(\psi \partial_x F_1 + \frac 12 \psi' F_1\right) \right| \lesssim 
|z|\left( |z|^2 + \|v_1\|_{L^2_{\omega}}^2\right) \left(|z|^2+  \|v_2\|_{L^2_{\omega}}\right)  .
\end{equation}
By \eqref{decaygg} and \eqref{allsmall}, we obtain
\begin{equation}\label{SFtrois}
\left|\alpha \int g F_2 \right| \lesssim 
|z|^2  \left( |z|^3 +|z|  \|v_1\|_{L^2_{\omega}}+\|v_1\|_{L^2_{\omega}}^2\right) .
\end{equation}
Now, we address the only remaining term in \eqref{virial error}, which is $\int v_1 \left(\psi \partial_x F_2+\frac 12 \psi' F_2\right)$.
We decompose
\[
F_2 = \widetilde F_2  - 3 H v_1^2 - v_1^3.
\]
Observe that $\widetilde F_2$  contains only terms with an explicit decay $e^{-\frac{|x|}{\sqrt{2}}}$ (coming from $Y_1(x)$, $q(x)$ and their derivatives).
Thus, by similar arguments as before,
\begin{equation}\label{SFquatre}
\left| \int v_1 \left(\psi \partial_x \widetilde F_2+\frac 12 \psi' \widetilde F_2\right)  \right| \lesssim 
  \|v_1\|_{L^2_{\omega}} \left( |z|^3 +|z|  \|v_1\|_{L^2_{\omega}} + \|v_1\|_{L^2_{\omega}}^2\right) .
\end{equation}

We are now reduced to controlling the term
\[
- \int v_1 \left( \psi \partial_x \left(3 H v_1^2 + v_1^3\right) +\frac 12 \psi'  \left(3 H v_1^2 + v_1^3\right) \right)
= \int  \left(\psi \partial_x v_1 + \frac 12\psi' v_1\right) \left(3 H v_1^2 + v_1^3\right) .
\]
Integrating by parts,
\[
 3 \int  \left(\psi \partial_x v_1 + \frac 12 \psi' v_1\right)  H v_1^2  
 = \int \left( \frac 12\psi' H - \psi H'\right) v_1^3.
\]
We claim 
\begin{equation}\label{SFcinq}
\int  (\psi'+\psi H') |v_1|^3\lesssim \|v_1\|_{L^\infty} \|\partial_x w\|_{L^2}^2 \lesssim \varepsilon \|\partial_x w\|_{L^2}^2.
\end{equation}
Indeed, by parity, the definition of $\psi$ \eqref{defpsi} and $w$ \eqref{def w}, and the decay property of $H'$, we have (with $\lambda=8$)
\begin{align*}
\int  (\psi'+\psi H') |v_1|^3
\lesssim   \int_0^{+\infty} e^{-\frac {2 x}{{\lambda } \sqrt{2} }} |v_1|^3
\lesssim   \int_0^{+\infty} e^{\frac x{\lambda \sqrt{2}}} |w|^3.
\end{align*}
Then, integrating by parts, using $w(0)=0$ (the function $w$ is odd)
\begin{align*}
\int_0^{+\infty} e^{\frac x{\lambda \sqrt{2}}} |w|^3
&= - \lambda\sqrt{2}  \int_0^{+\infty} e^{\frac x{\lambda \sqrt{2}}} \partial_x(|w|^3)
= - 3 \lambda \sqrt{2}  \int_0^{+\infty} e^{\frac x{\lambda \sqrt{2}}} (\partial_x w) w |w|\\
&\leq 6\lambda \|v_1\|_{L^\infty}^{\frac 12} \int_0^{+\infty} e^{\frac x{2 \lambda \sqrt{2}}} |\partial_x w| |w|^{\frac 32}
\leq 18 \lambda^2 \|v_1\|_{L^\infty} \int_0^{+\infty} |\partial_x w|^2 +\frac 14  \int_0^{\infty} e^{\frac x{\lambda \sqrt{2}}} |w|^3.
\end{align*}
Thus,
\[
\int_0^{+\infty} e^{\frac x{\lambda \sqrt{2}}} |w|^3
\lesssim  \|v_1\|_{L^\infty} \|\partial_x w\|_{L^2}^2,
\]
and \eqref{SFcinq} is proved.

Finally, we have by integration by parts
\[
\int  \left(\psi \partial_x v_1 + \frac 12 \psi' v_1\right)   v_1^3 = \frac 14 \int \psi' v_1^4 \geq 0.
\]
This term happens to have the right sign for   estimate \eqref{10un} (this is related to the fact that equation \eqref{wave ac} is defocusing), but we can also bound this term in absolute value since by \eqref{SFcinq},
\begin{equation}\label{SFsix}
\int \psi' v_1^4\lesssim \varepsilon \int \psi' |v_1|^3\lesssim \varepsilon^2 \|\partial_x w\|_{L^2}^2.
\end{equation}

In conclusion, \eqref{lll} is a consequence of \eqref{rwv}, \eqref{SFun}, \eqref{SFdeux}, \eqref{SFtrois}, \eqref{SFquatre}, \eqref{SFcinq} and \eqref{SFsix}. 
\medskip

{\bf Step 5.} Proof of \eqref{10trois}. We use \eqref{Itwo} with $\psi'(x) = \sech\left(\frac x{2 \sqrt{2}}\right) $.
Note that
\begin{align*}
\int \sech\left(\frac x{2 \sqrt{2}}\right)\left(|\partial_x v_1|^2+ v_1^2\right)
+ \int \left| \sech''\left(\frac x{2 \sqrt{2}}\right)\right|v_1^2 \lesssim \|v_1\|^2_{H^1_{\omega}}, 
\end{align*}
\[
\left|\alpha \int  \sech\left(\frac x{2 \sqrt{2}}\right) v_1 f\right|\lesssim |z|^2 \|v_1\|_{H^1_{\omega}},
\]
and 
\[
\left|\int  \sech\left(\frac x{2 \sqrt{2}}\right) \left(|F_1 v_2| + |v_1 F_2|\right)  \right|
\lesssim   |z|^3\|v_2\|_{L^2_{\omega}}+ \|v_1\|_{H^1_{\omega}}^2.
\]
Using these estimates, we obtain from \eqref{Itwo}
\[
\frac d{dt} \int \sech\left(\frac x{2 \sqrt{2}}\right) v_1 v_2 \geq    \|v_2\|_{L^2_{\omega}}^2
- C \|v_1\|_{H^1_{\omega}}^2  - C |z|^3 \|v_2\|_{L^2_{\omega}}  
+ C |z|^4,
\]
and \eqref{10trois} follows.
\end{proof}

\subsection{Conclusion. Proof of (\ref{Conclusion_0})}\label{Final} 
Let
\begin{equation}\label{defH}
\mathcal H:= \int \left((\partial_x v_1)^2 + 2v_1^2 + v_2^2\right)\sech\left(\frac x{2 \sqrt{2}}\right) .
\end{equation}
Then, using \eqref{syd}, we have
\begin{equation}\label{Kone}\begin{aligned}
\dot{\mathcal  H}  &= 2 \int  \sech\left(\frac x{2 \sqrt{2}}\right) \left((\partial_x \dot v_1) (\partial_x v_1) + 2\dot v_1 v_1  + \dot v_2 v_2\right)\\
& = 2 \int \sech\left(\frac x{2 \sqrt{2}}\right) \left( (\partial_x v_2)(\partial_x v_1) + 2 v_2 v_1  -  (\mathcal L v_1)v_2  -\alpha   fv_2 \right)\\
&+ 2 \int  \sech\left(\frac x{2 \sqrt{2}}\right)\left((\partial_x F_1)(\partial_x v_1) + 2 F_1 v_1 + F_2 v_2\right)\\
& = -\frac 1{\sqrt{2}} \int \sech'\left(\frac x{2 \sqrt{2}}\right)   v_2(\partial_x v_1)
+ 2 \int \sech\left(\frac x{2 \sqrt{2}}\right) \left(3(1-H^2) v_1 v_2 -\alpha   fv_2 \right) \\
&+ 2 \int  \sech\left(\frac x{2 \sqrt{2}}\right)\left((\partial_x F_1)(\partial_x v_1) + 2 F_1 v_1 + F_2 v_2\right)
\end{aligned}\end{equation}
Note that 
\begin{equation}\label{Ktwo}
\left| \int \sech'\left(\frac x{2 \sqrt{2}}\right)   v_2(\partial_x v_1)\right| 
\lesssim \int \left( (\partial_x v_1)^2+{v_2^2} \right) \sech\left(\frac x{2 \sqrt{2}}\right). 
\end{equation}
 
From \eqref{11un}, there exists a sequence $t_n\to +\infty$ such that $\mathcal H(t_n)+z(t_n)\to 0$.
From \eqref{Kone}, \eqref{Ktwo} and the estimates on $F_1$ and $F_2$, we have
\[
|\dot {\mathcal H}|  \lesssim  |z(t)|^4  + \|v(t)\|_{H^1_{\omega}\times L^2_{\omega}}^2 .
\]
Let $t\in \RR$, integrating on $[t,t_n]$ and passing to the limit as $n\to +\infty$, we obtain
\[
\mathcal H(t) \lesssim \int_{t}^{+\infty} \left(|z(t)|^4  + \|v(t)\|_{H^1_{\omega}\times L^2_{\omega}}^2 \right) dt.
\]
From \eqref{11un}, it follows that $\lim_{t\to +\infty} \mathcal H(t)=0.$ The same holds for $t\to -\infty$.
Thus, $\lim_{t\to \pm \infty} \|v\|_{H^1_{\omega}\times L^2_{\omega}}=0.$ Moreover, from \eqref{dz2} and \eqref{Fab}, we have
\begin{equation}\label{dz4}
\left|\frac d{dt}  |z|^4  \right|=
2|\alpha  F_\alpha+\beta F_\beta | \lesssim |z|^3 \left( |z|^2 + \|v_1\|^2_{L^2_{\omega}} \right)\lesssim |z|^4+\|v_1\|^2_{L^2_{\omega}}.
\end{equation}
Using \eqref{11un} and   similar arguments as before, we obtain $\lim_{t\to \pm\infty} |z(t)|=0$. Therefore, (\ref{Conclusion_0}) follows from \eqref{defv}.


\providecommand{\bysame}{\leavevmode\hbox to3em{\hrulefill}\thinspace}
\providecommand{\MR}{\relax\ifhmode\unskip\space\fi MR }
\providecommand{\MRhref}[2]{%
  \href{http://www.ams.org/mathscinet-getitem?mr=#1}{#2}
}
\providecommand{\href}[2]{#2}

\end{document}